\documentclass[final]{siamart190516}
\newtheorem{remark}{Remark}[section]

\usepackage[all]{xy}
\usepackage{amssymb}
\usepackage{braket,amsfonts,amsopn}
\usepackage{bm,esint}
\usepackage{makecell,multirow,diagbox}
\usepackage{mathrsfs}
\usepackage{graphicx,epstopdf}
\usepackage{stmaryrd}
\usepackage[caption=false]{subfig} 
\numberwithin{theorem}{section}
\numberwithin{equation}{section}
\numberwithin{algorithm}{section}

\newcommand{\lr}[1]{\llbracket#1\rrbracket}

\title{Fast auxiliary space preconditioners on surfaces}

\begin{document}


  \author{Yuwen Li%
         \thanks{Department of Mathematics, The Pennsylvania State University, University Park, PA 16802 {Email:} yuwenli925@gmail.com}
}
  \maketitle

  \begin{abstract}
  This work presents uniform preconditioners for the discrete Laplace--Beltrami operator on hypersurfaces. In particular, within the framework of fast auxiliary space preconditioning (FASP), we develop efficient and user-friendly multilevel preconditioners for the Laplace--Beltrami type equation discretized by Lagrange, nonconforming linear, and discontinuous Galerkin elements. The analysis applies to semi-definite problems on a closed surface. Numerical experiments on 2d surfaces and 3d hypersurfaces are presented to illustrate the efficiency of the proposed preconditioners.
  \end{abstract}
  \begin{AMS}{65N30, 65N55, 65F08}\end{AMS}
  \pagestyle{myheadings}
  \thispagestyle{plain}
  \markboth{Y.~Li}{FASP on surfaces}


\section{Introduction}
Discretizations of partial differential equations (PDEs) often lead to sparse large-scale and ill-conditioned algebraic systems of linear equations. Those discrete linear systems should be solved by well-designed fast linear solvers otherwise the computational cost would be unacceptable. In practice, multilevel iterative solvers are among the most efficient and popular linear solvers for discretized PDEs. In addition, the performance of these multilevel solvers could be improved when used in Krylov subspace methods as preconditioners. On Euclidean domains, the theory of multilevel methods is well established, see, e.g., \cite{Brandt1977,BankDupont1981,Hackbusch1985,BramblePasciakWangXu1991,Xu1992,XuZikatanov2002}.

In recent decades, numerical methods for solving PDEs on surfaces has been a popular and important research area, see \cite{DeckelnickDziukElliott2005,DziukElliott2013} and references therein for an introduction. To efficiently implement numerical PDEs on surfaces, fast surface linear solvers are indispensable. In contrast to planar multigrid, the hierarchy of triangulated surfaces is never nested. As a result, solving algebraic linear systems from surfaces is often more challenging. In fact, surface multilevel iterative methods are still in the early stage and many fundamental questions have not been addressed yet. In \cite{AksoyluKhodakovskySchroder2005}, a surface multigrid method is developed based on a surface mesh coarsening algorithm. The work \cite{MaesKunothBultheel2007} analyzes the BPX multigrid for elliptic equations on spheres. The analysis of the hierarchical basis multigrid method on two dimensional surfaces could be found in \cite{KornhuberYserentant2008}. For surface linear elements, \cite{BonitoPasciak2012} presents a general analysis of multilevel methods including the multigrid V-cycle. A multigrid solver for the closest point finite difference scheme on surfaces is presented in \cite{ChenMacdonald2015}. Besides the aforementioned geometric multigrid methods, the discrete surface PDEs could also be directly solved by the algebraic multigrid (AMG), cf.~\cite{RugeStuben1987,BrandtMcCormickRuge1985,WanChanSmith1999,BankSmith2002,XuZikatanov2017}.

In this work, we first consider the second order elliptic equation on a hypersurface $\mathcal{M}$ in $\mathbb{R}^{d+1}$ ($d$ is a positive integer), which is discretized by the surface linear element. To derive surface  geometric multigrid, we utilize a sequence of piecewise flat hypersurfaces $\{\mathcal{M}_j\}_{j=0}^J$, which approximates $\mathcal{M}$ from the coarsest $\mathcal{M}_0$ to the finest level $\mathcal{M}_J$. 
Let $\mathcal{T}_j$ be the set of $d$-dimensional faces of $\mathcal{M}_j,$ and $V_j$ the space of globally continuous and piecewise linear polynomials on $\mathcal{M}_j$ with respect to $\mathcal{T}_j$. Although not nested in the classical sense, $\{V_j\}_{j=0}^J$ is still logically nested via certain natural intergrid transfer operator. From this point of view, classical BPX preconditioner, hierarchical basis multigrid, and standard multigrid were constructed and implemented in \cite{AksoyluKhodakovskySchroder2005,MaesKunothBultheel2007,KornhuberYserentant2008,BonitoPasciak2012}. In \cite{BonitoPasciak2012}, rigorous convergence analysis of surface multigrids was done by lifting $\{\mathcal{T}_j\}_{j=0}^J$ to $\mathcal{M}$ and then applying the classical multigrid analysis. Due to its perturbation nature, the analysis in \cite{BonitoPasciak2012} assumes that the mesh size of $\mathcal{T}_0$ is small enough.
  
For surface Lagrange elements, we propose a new multilevel approach based on fast auxiliary space preconditioning (FASP) in \cite{Xu1996}. In particular, the auxiliary space is the space of continuous and piecewise linear elements on the initial surface $\mathcal{M}_0$, which is fixed and piecewise affine. The auxiliary transfer operator is available and bi-Lipschitz under common assumptions used in surface finite element literature. Since $\mathcal{M}_0$ is a fixed polytope with flat faces, uniform preconditioners on $\mathcal{M}_0$ directly follow from  the nested grid hierarchy on $\mathcal{M}_0$ and the multilevel theory on Euclidean domains. Then a combination of the preconditioning result on $\mathcal{M}_0$ and the auxiliary inter-surface operator yields a preconditioner for the discrete problem on $\mathcal{M}_h$. The analysis is new and independent of the assumption on the small mesh size of $\mathcal{M}_0$. On a closed surface, our approach preconditions the semi-definite Laplace--Beltrami operator  by a positive definite operator on the reference surface $\mathcal{M}_0.$

In addition, the proposed FASP approach leads to efficient preconditioners for the Crouzeix--Raviart (CR) element and discontinuous Galerkin (DG) method on the approximate surface $\mathcal{M}_h$. We use the conforming linear element space on $\mathcal{M}_h$ as the auxiliary space and the inclusion mapping as the transfer operator. In other words, the resulting linear solver is a two-level method using linear nodal elements in the coarse level, which could be further approximated by the established geometric multigrid. If the coarse solve is simply  replaced with AMG, we obtain semi-analytic algebraic CR and DG solvers on surfaces without using any grid hierarchy. For CR and DG discretizations on surfaces, such two-level FASP solvers outperform the direct AMG solvers in several numerical experiments. For DG methods on Euclidean domains, the two-level auxiliary space preconditioners could be found in e.g. \cite{DobrevLazarovVassilevskiZikatanov2006,AyusoZikatanov2009}.

The rest of this paper is organized as follows. In Section \ref{secabstract}, we introduce auxiliary space lemmas and the preconditioned conjugate gradient method in the Hilbert space. In Section \ref{seclinear}, we develop auxiliary space preconditioners for the Laplace--Beltrami type equation discretized by conforming linear elements. Section \ref{secNCDG} is devoted to the FASP solvers for the CR and DG discretizations. The proposed preconditioners are tested in several numerical experiments in Section \ref{secNE}. Possible extensions to higher order methods are discussed in Section \ref{secconclusion}.

\section{Abstract framework}\label{secabstract}

For a Hilbert space $\mathcal{V}$, let $(\cdot,\cdot)_\mathcal{V}$ denote its inner product, $\|\cdot\|_\mathcal{V}$ the $\mathcal{V}$-norm, $\mathcal{V}^\prime$ the dual space of $\mathcal{V},$ $I=I_\mathcal{V}$ the identity mapping on $\mathcal{V}$, and $\langle\cdot,\cdot\rangle=\langle\cdot,\cdot\rangle_{\mathcal{V}'\times\mathcal{V}}$ the action of $\mathcal{V}'$ on $\mathcal{V}$. For $g_1: \mathcal{V}\rightarrow\mathcal{V}^\prime$ and $g_2: \mathcal{V}^\prime\rightarrow\mathcal{V}$, let $g_1^t: \mathcal{V}\rightarrow\mathcal{V}^\prime$ and $g_2^t: \mathcal{V}^\prime\rightarrow\mathcal{V}$ be defined as
\begin{align*}
        &\langle g_1^tv_1,v_2\rangle=\langle g_1v_2,v_1\rangle,\quad\forall v_1, v_2\in\mathcal{V},\\
    &\langle r_1, g_2^tr_2\rangle=\langle r_2, g_2r_1\rangle,\quad\forall r_1, r_2\in\mathcal{V}^\prime.
\end{align*} 
Given a linear operator $g: \mathcal{V}_1\rightarrow\mathcal{V}_2$, let $R(g)$ denote its range, $N(g)$ the kernel of $g$, and $g^\prime: \mathcal{V}_2^\prime\rightarrow\mathcal{V}_1^\prime$ the adjoint of $g$, i.e.,
\begin{equation*}
        \langle g_1^\prime r,v\rangle=\langle r,gv\rangle,\quad\forall v\in\mathcal{V}_1,~\forall r\in\mathcal{V}_2^\prime.
\end{equation*}
For a bounded linear operator $\mathcal{A}: \mathcal{V}\rightarrow\mathcal{V}^\prime$, we say it is symmetric and positive semi-definite (SPSD) provided $\mathcal{A}=\mathcal{A}^t$ and $\langle\mathcal{A}v,v\rangle\geq0~\forall v\in\mathcal{V}.$
The SPSD operator $\mathcal{A}$ defines a bilinear form and a semi-norm on $\mathcal{V}$ by 
\begin{equation*}
    (v,w)_\mathcal{A}:=\langle\mathcal{A}v,w\rangle,\quad|v|^2_{\mathcal{A}}:=\langle\mathcal{A}v,v\rangle\quad\forall v,w\in\mathcal{V}.
\end{equation*}
We say $\mathcal{A}: \mathcal{V}\rightarrow\mathcal{V}^\prime$ is symmetric and positive definite (SPD) provided $\mathcal{A}$ is SPSD and $N(\mathcal{A})=\{0\}$. If $\mathcal{A}$ is SPD, $(\cdot,\cdot)_\mathcal{A}$ is an inner product, $|\cdot|_\mathcal{A}$ is a norm, and the notation $|\cdot|_\mathcal{A}$ is  replaced  with $\|\cdot\|_\mathcal{A}.$ For the kernel $\mathcal{N}:=N(\mathcal{A})$ and $v\in\mathcal{V}$, we define 
\begin{equation*}
    [v]:=\big\{w\in\mathcal{V}: w-v\in\mathcal{N}\big\},
\end{equation*}
the equivalence class containing $v$, and the quotient Hilbert space 
\begin{equation*}
     \mathcal{V}/\mathcal{N}:=\big\{[v]: v\in\mathcal{V}\big\}.
\end{equation*}
It is noted that the dual space of the quotient space $\mathcal{V}/\mathcal{N}$ could be defined as  a subspace of $\mathcal{V}^\prime$ as follows
\begin{equation*}
    (\mathcal{V}/\mathcal{N})^\prime:=\big\{r\in\mathcal{V}^\prime: \langle r,v\rangle=0~\forall v\in\mathcal{N}\big\}.
\end{equation*}
In the literature, $(\mathcal{V}/\mathcal{N})^\prime$ called is the polar set of $\mathcal{N}$, see, e.g., \cite{GiraultRaviart1986,BoffiBrezziFortin2013}. The next elementary lemma characterizes  $(\mathcal{V}/\mathcal{N})^\prime$ and the proof is included for completeness.
\begin{lemma}\label{rangeA}
Let $\mathcal{A}: \mathcal{V}\rightarrow\mathcal{V}^\prime$ be a SPSD operator with the kernel $\mathcal{N}=N(\mathcal{A})$ and a closed range. Then 
we have
\begin{equation*}
    R(\mathcal{A})=(\mathcal{V}/\mathcal{N})^\prime.
\end{equation*}
\end{lemma}
\begin{proof}
Let $\mathcal{B}: \mathcal{V}^\prime\rightarrow\mathcal{V}$ be the Riesz representation. The range of $\mathcal{B}\mathcal{A}: \mathcal{V}\rightarrow\mathcal{V}$ is closed and $\mathcal{B}\mathcal{A}$ is symmetric with respect to $(\cdot,\cdot)_{\mathcal{B}^{-1}}$. Let $\mathcal{W}^{\perp\mathcal{B}^{-1}}$ be the orthogonal complement of  a subspace $\mathcal{W}$ with respect to $(\cdot,\cdot)_{\mathcal{B}^{-1}}$ in $\mathcal{V}$. Then using $R(\mathcal{B}\mathcal{A})=N(\mathcal{B}\mathcal{A})^{\perp\mathcal{B}^{-1}}$ (by the closed range theorem) and $N(\mathcal{B}\mathcal{A})=N(\mathcal{A})$, we obtain
\begin{align*}
    R(\mathcal{A})&=\mathcal{B}^{-1}R(\mathcal{B}\mathcal{A})=\mathcal{B}^{-1}\left(N(\mathcal{B}\mathcal{A})^{\perp\mathcal{B}^{-1}}\right)=\mathcal{B}^{-1}(N(\mathcal{A})^{\perp\mathcal{B}^{-1}})\\
    &=\mathcal{B}^{-1}\big(\big\{v\in\mathcal{V}: (v,z)_{\mathcal{B}^{-1}}=0~\forall z\in\mathcal{N}\big\}\big)\\
    &=\big\{r\in\mathcal{V}^\prime: \langle r,z\rangle=0~\forall z\in\mathcal{N}\big\},
\end{align*}
which completes the proof.
\end{proof}

\subsection{Semi-definite conjugate gradient}
Given a SPSD operator $\mathcal{A}: \mathcal{V}\rightarrow\mathcal{V}^\prime$ and $f\in\mathcal{V}^\prime$, we consider the operator equation
\begin{equation}\label{Auf}
    \mathcal{A}u=f.
\end{equation}
It follows from Lemma \ref{rangeA} that \eqref{Auf} has a solution in $\mathcal{V}$ provided the compatibility condition $f\in(\mathcal{V}/\mathcal{N})^\prime$ holds. In this case, the solution to \eqref{Auf} is uniquely determined modulo $\mathcal{N}$. In practice, the preconditioned conjugate gradient (PCG) method is among the most successful iterative algorithms for solving \eqref{Auf} in finite-dimensional spaces. Let $\mathcal{B}: \mathcal{V}^\prime\rightarrow\mathcal{V}$ be a SPD operator. 
If \eqref{Auf} arises from the discretization of positive definite PDEs, the PCG would uniformly converge provided the condition number $\kappa(\mathcal{B}\mathcal{A})$ is uniformly bounded with respect to discretization parameters such as the mesh size. In this case, $\mathcal{B}$ is said to be a preconditioner for $\mathcal{A}.$ For theoretical convenience, we write the PCG method for solving \eqref{Auf} in the Hilbert space, see Algorithm \ref{PCG}.
\begin{algorithm}[H]\label{PCG}
\caption{Preconditioned  conjugate gradient (PCG)}
Input: An initial guess $u_0\in\mathcal{V}$ and an error tolerance $\textsf{tol}>0$.

Initialize: $r_0 = f-\mathcal{A}u_0$, $z_0=\mathcal{B} r_0$, $p_0 = z_0$, and $k=0$;

\textbf{While}{$\langle r_k,z_k\rangle\geq\textsf{tol}$}

$\qquad\alpha_k=\frac{\langle r_k,z_k \rangle}{|p_k|_\mathcal{A}^2}$;

$\qquad u_{k+1} =  u_k + \alpha_k p_k$; 

$\qquad r_{k+1} =  r_k - \alpha_k\mathcal{A}p_k$; 

$\qquad z_{k+1} =  \mathcal{B} r_{k+1}$; 

$\qquad \beta_k =  \frac{\langle r_{k+1},z_{k+1}\rangle}{\langle r_k,z_k\rangle}$;

$\qquad p_{k+1} = z_{k+1} + \beta_k p_{k}$; 

$\qquad k=k+1$;

\textbf{EndWhile}
\end{algorithm}
From the above algorithm,
it is observed that $r_k\in(\mathcal{V}/\mathcal{N})^\prime$ for each $k$. Therefore we have
$z_k\perp_{\mathcal{B}^{-1}}\mathcal{N}$ and $p_k\perp_{\mathcal{B}^{-1}}\mathcal{N}$. It then follows that before the iteration stops, $|p_k|_{\mathcal{A}}\neq0$ and Algorithm \ref{PCG} is well defined for a SPSD operator $\mathcal{A}.$

The semi-definite operator
$\mathcal{A}$ induces the linear operator $[\mathcal{A}]: \mathcal{V}/\mathcal{N}\rightarrow(\mathcal{V}/\mathcal{N})^\prime$ by
\begin{equation*}
    \langle[\mathcal{A}][v],w\rangle:=\langle\mathcal{A}v,w\rangle,\quad\forall [v]\in\mathcal{V}/\mathcal{N},~\forall w\in \mathcal{V}.
\end{equation*}
Clearly $[\mathcal{A}]$ is well-defined and $[\mathcal{A}]$ is SPD by Lemma \ref{rangeA}. Similarly, a preconditioner $\mathcal{B}: \mathcal{V}^\prime\rightarrow\mathcal{V}$ defines a linear operator $[\mathcal{B}]: (\mathcal{V}/\mathcal{N})^\prime\rightarrow\mathcal{V}/\mathcal{N}$ by
\begin{equation*}
    [\mathcal{B}]r:=[\mathcal{B}r],\quad\forall r\in (\mathcal{V}/\mathcal{N})^\prime.
\end{equation*}
The next lemma shows that $[\mathcal{B}]$ is also SPD.
\begin{lemma}\label{SPDB}
Let $\mathcal{B}: \mathcal{V}^\prime\rightarrow\mathcal{V}$ be SPD. Then $[\mathcal{B}]: (\mathcal{V}/\mathcal{N})^\prime\rightarrow\mathcal{V}/\mathcal{N}$ is SPD.
\end{lemma}
\begin{proof}
Clearly $[\mathcal{B}]$ is SPSD. Given $r\in (\mathcal{V}/\mathcal{N})^\prime,$
$[\mathcal{B}]r=0$ in $\mathcal{V}/\mathcal{N}$ implies $\mathcal{B}r\in\mathcal{N}$ and thus $\langle r,\mathcal{B}r\rangle=0$. Because $\mathcal{B}$ is SPD, we have $r=0$ and $N([\mathcal{B}])=0$. 
\end{proof}

Throughout the rest of this paper,
the condition number $\kappa(\mathcal{B}\mathcal{A})$ is defined as the ratio between the maximum and minimum eigenvalues of $[\mathcal{B}][\mathcal{A}]: \mathcal{V}/\mathcal{N}\rightarrow\mathcal{V}/\mathcal{N}.$ 
Replacing $\mathcal{A}, \mathcal{B}, z_k, p_k$ in  Algorithm \ref{PCG} with $[\mathcal{A}], [\mathcal{B}], [z_k], [p_k]$, we obtain a PCG algorithm for the SPD problem $[\mathcal{A}][u]=f$ in the quotient space  $\mathcal{V}/\mathcal{N}.$
Therefore classical convergence analysis of PCG (cf.~\cite{Xu1992}) implies
\begin{equation*}
    \big\|[u]-[u_k]\big\|_{[\mathcal{A}]}\leq2\left(\frac{\kappa([\mathcal{B}][\mathcal{A}])-1}{\kappa([\mathcal{B}][\mathcal{A}])+1}\right)^k\big\|[u]-[u_0]\big\|_{[\mathcal{A}]},
\end{equation*}
or equivalently
\begin{equation*}
    |u-u_k|_{\mathcal{A}}\leq2\left(\frac{\kappa({\mathcal{B}}{\mathcal{A}})-1}{\kappa({\mathcal{B}}{\mathcal{A}})+1}\right)^k|u-u_0|_{\mathcal{A}}
\end{equation*}

Other existing theoretical discussions on iterative methods for singular and nearly singular PDEs could be found in e.g., \cite{BochevLehoucq2005,LeeWuXuZikatanov2007,LeeWuXuZikatanov2008,AyusoBrezziMariniXuZikatanov2014}. We develop the above framework because of theoretical convenience when analyzing the semi-definite problems considered in this paper.

\subsection{Auxiliary space preconditioning}
The following fictitious space lemma \cite{Nepomnyaschikh1992} turns out to be a powerful tool for estimating the conditioner number $\kappa(\mathcal{B}\mathcal{A})$ and developing uniform preconditioners, see, e.g., \cite{Xu1996,HiptmairXu2007,Chen2011}.
\begin{lemma}[Fictitious space lemma]\label{FSP}
Let $\mathcal{V}$, $\widetilde{\mathcal{V}}$ be Hilbert spaces and $\mathcal{A}: \mathcal{V}\rightarrow \mathcal{V}^\prime$, $\widetilde{\mathcal{A}}: \widetilde{\mathcal{V}}\rightarrow \widetilde{\mathcal{V}}^\prime$, $\widetilde{\mathcal{B}}: \widetilde{\mathcal{V}}^\prime\rightarrow \widetilde{\mathcal{V}}$ be SPD operators.
Assume $\Pi: \widetilde{\mathcal{V}}\rightarrow\mathcal{V}$ is a surjective linear operator, and 
\begin{itemize}
\item There exists a constant $c_0>0$ such that  $\|\Pi\tilde{v}\|_{\mathcal{A}}\leq c_0\|\tilde{v}\|_{\widetilde{\mathcal{A}}}$ for each $\tilde{v}\in \widetilde{\mathcal{V}};$
    \item 
There exists a constant $c_1>0$ such that given any $v\in\mathcal{V},$ some $\tilde{v}\in \widetilde{\mathcal{V}}$ satisfies 
$$\Pi\tilde{v}=v,\quad\|\tilde{v}\|_{\widetilde{\mathcal{A}}}\leq c_1\|v\|_{\mathcal{A}}.$$
\end{itemize}
Then for $\mathcal{B}=\Pi\widetilde{\mathcal{B}}\Pi^\prime: \mathcal{V}^\prime\rightarrow\mathcal{V}$ we have
\begin{equation*}
    \kappa(\mathcal{B}\mathcal{A})\leq \left(c_0c_1\right)^2\kappa(\widetilde{\mathcal{B}}\widetilde{\mathcal{A}}).
\end{equation*}
\end{lemma}

Lemma \ref{FSP} directly yields a condition number
estimate for semi-definite operators.
\begin{corollary}\label{FSP2}
Let $\mathcal{V}$, $\widetilde{\mathcal{V}}$ be Hilbert spaces, $\widetilde{\mathcal{B}}: \widetilde{\mathcal{V}}^\prime\rightarrow \widetilde{\mathcal{V}}$ be SPD, and $\mathcal{A}: \mathcal{V}\rightarrow \mathcal{V}^\prime$, $\widetilde{\mathcal{A}}: \widetilde{\mathcal{V}}\rightarrow \widetilde{\mathcal{V}}^\prime$ be SPSD operators with closed ranges and kernels $\mathcal{N}=N(\mathcal{A})$, $\widetilde{\mathcal{N}}=N(\widetilde{\mathcal{A}})$.
Assume $\Pi: \widetilde{\mathcal{V}}\rightarrow\mathcal{V}$ is a linear operator, and 
\begin{itemize}
\item $\Pi$ preserves kernels: $\Pi(\widetilde{\mathcal{N}})\subseteq\mathcal{N};$
\item There exists a constant $c_0>0$ such that  $|\Pi\tilde{v}|_{\mathcal{A}}\leq c_0|\tilde{v}|_{\widetilde{\mathcal{A}}}$ for each $\tilde{v}\in \widetilde{\mathcal{V}};$
    \item 
There exists a constant $c_1>0$ such that for any $v\in\mathcal{V},$ some $\tilde{v}\in \widetilde{\mathcal{V}}$ satisfies 
$$\Pi\tilde{v}-v\in\mathcal{N},\quad|\tilde{v}|_{\widetilde{\mathcal{A}}}\leq c_1|v|_{\mathcal{A}}.$$
\end{itemize}
Then for $\mathcal{B}=\Pi\widetilde{\mathcal{B}}\Pi^\prime: \mathcal{V}^\prime\rightarrow\mathcal{V}$ we have
\begin{equation*}
    \kappa({\mathcal{B}}{\mathcal{A}})\leq \left(c_0c_1\right)^2\kappa({\widetilde{\mathcal{B}}}{\widetilde{\mathcal{A}}}).
\end{equation*}
\end{corollary}
\begin{proof}
Using $\Pi(\widetilde{\mathcal{N}})\subseteq \mathcal{N},$ we obtain a well-defined operator $[\Pi]: \widetilde{\mathcal{V}}/\widetilde{\mathcal{N}}\rightarrow\mathcal{V}/\mathcal{N}$ given by $[\Pi][\tilde{v}]=[\Pi\tilde{v}]$ $\forall[\tilde{v}]\in\widetilde{\mathcal{V}}/\widetilde{\mathcal{N}}$. It follows from Lemmas \ref{rangeA} and \ref{SPDB} that the operators $[\mathcal{A}]: \mathcal{V}/\mathcal{N}\rightarrow(\mathcal{V}/\mathcal{N})^\prime$, $[\widetilde{\mathcal{A}}]: \widetilde{\mathcal{V}}/\widetilde{\mathcal{N}}\rightarrow(\widetilde{\mathcal{V}}/\widetilde{\mathcal{N}})^\prime$, $[\widetilde{\mathcal{B}}]: (\widetilde{\mathcal{V}}/\widetilde{\mathcal{N}})^\prime\rightarrow\widetilde{\mathcal{V}}/\widetilde{\mathcal{N}}$ are all SPD. Then we could finish the proof by noticing  $[\mathcal{B}]=[\Pi][\widetilde{\mathcal{B}}][\Pi]^\prime$ and using Lemma \ref{FSP} with $\mathcal{V}$, $\widetilde{\mathcal{V}}$, $\mathcal{A}$, $\widetilde{\mathcal{A}}$, $\widetilde{\mathcal{B}}$, $\Pi$ replaced by $\mathcal{V}/\mathcal{N}$, $\widetilde{\mathcal{V}}/\widetilde{\mathcal{N}}$, $[\mathcal{A}]$, $[\widetilde{\mathcal{A}}]$, $[\widetilde{\mathcal{B}}]$, $[\Pi]$.
\end{proof}
It follows from the above
corollary that a preconditioner $\widetilde{\mathcal{B}}$ for the semi-definite $\widetilde{\mathcal{A}}$ leads to a preconditioner    $\mathcal{B}$ for the semi-definite operator $\mathcal{A}$. Corollary \ref{FSP2} is the only tool used for preconditioning singular operators on surfaces.

Let $V$ be a subspace of $\mathcal{V}$ and $\mathcal{I}$ the inclusion from $V$ to $\mathcal{V}.$ Using Lemma \ref{FSP} with $\widetilde{\mathcal{V}}=\mathcal{V}_{\|\cdot\|_{\mathcal{S}^{-1}}}\times V$, $\Pi=(I_\mathcal{V},\mathcal{I}): \widetilde{\mathcal{V}}\rightarrow\mathcal{V}$, we obtain a two-level additive preconditioner in Lemma \ref{ASP} for the norm of $\mathcal{V}$, which is a special case of the auxiliary space lemma proposed in \cite{Xu1996} for preconditioning a wide range of discrete problems.
\begin{lemma}[Additive two-level preconditioner]\label{ASP}
Let $V\subset\mathcal{V}$ be  Hilbert spaces,  $\mathcal{A}: \mathcal{V}\rightarrow\mathcal{V}^\prime$, $A: V\rightarrow V^\prime$, and $B: V^\prime\rightarrow V$, $\mathcal{S}: \mathcal{V}^\prime\rightarrow\mathcal{V}$ be SPD. 
Let $\mathcal{I}: V\rightarrow\mathcal{V}$ be the inclusion, $\mathcal{Q}=\mathcal{I}^\prime: \mathcal{V}^\prime\rightarrow V^\prime,$ and $P: \mathcal{V}\rightarrow V$ be a linear operator. Assume 
\begin{itemize}
\item There exists a constant $c_0>0$ such that $\|v\|_{\mathcal{A}}\leq c_0\|v\|_{A}$ for each $v\in V$;
\item There exists constants $c_s, c_1, c_2>0$ such that for each $v\in\mathcal{V},$
\begin{align*}
    \|v\|_{\mathcal{A}}&\leq c_s\|v\|_{\mathcal{S}^{-1}},\\
    \|Pv\|_A&\leq c_1\|v\|_{\mathcal{A}},\\
    \|v-Pv\|_{\mathcal{S}^{-1}}&\leq c_2\|v\|_{\mathcal{A}}.
\end{align*}
\end{itemize}
Then the preconditioner $\mathcal{B}^a:=\mathcal{S}+B\mathcal{Q}$ satisfies
\begin{equation*}
    \kappa(\mathcal{B}^a\mathcal{A})\leq (c_1^2+c_2^2)(c_0^2+c_s^2)\kappa(BA).
\end{equation*}
\end{lemma}
Lemma \ref{ASP} only deals with SPD operators which are sufficient for our purpose. One could use Corollary \ref{FSP2} to obtain a generalized version of Lemma \ref{ASP} that is able to handle semi-definite operators.

Using the spaces $V$, $\mathcal{V}$ and operators $A, \mathcal{A}, \mathcal{S}, \mathcal{Q}$  in Lemma \ref{ASP}, we also obtain a two-level multiplicative method described in Algorithm \ref{twolevelSSC}.
\begin{algorithm}[H]
\caption{Multiplicative two-level preconditioner $\mathcal{B}^m$}\label{twolevelSSC}
Input: $u_0\in\mathcal{V}$ and $g\in\mathcal{V}^\prime$.

$u_1=u_0+\mathcal{S}(g-\mathcal{A}u_0)$;

$u_2=u_1+A^{-1}\mathcal{Q}(g-\mathcal{A}u_1)$;

$u_3=u_2+\mathcal{S}^t(g-\mathcal{A}u_2)$;

Output: $\mathcal{B}^mg:=u_3.$
\end{algorithm}
The next lemma addresses the convergence of $\mathcal{B}^m$ and the proof could be found in \cite{Zikatanov2008,XuZikatanov2017}.
\begin{lemma}[Multiplicative two-level method]\label{ASP2}
Let $V\subset\mathcal{V}$ be Hilbert spaces, $\mathcal{A}: \mathcal{V}\rightarrow\mathcal{V}^\prime$, $A: V\rightarrow V^\prime$, $\mathcal{S}: \mathcal{V}^\prime\rightarrow\mathcal{V}$ be SPD operators, and  $\bar{\mathcal{S}}=\mathcal{S}^t+\mathcal{S}-\mathcal{S}^t\mathcal{A}\mathcal{S}$ the symmetrization of $\mathcal{S}.$ 
Then $\mathcal{B}^m$ satisfies
\begin{equation*}
    \|I-\mathcal{B}^m\mathcal{A}\|_\mathcal{A}=1-\frac{1}{c},\quad c=\sup_{v\in\mathcal{V}, \|v\|_\mathcal{A}=1}\inf_{w\in \mathcal{V}}\|v-w\|_{\bar{\mathcal{S}}^{-1}}.
\end{equation*}
\end{lemma}
Lemma \ref{ASP2} states that $I-\mathcal{B}^m\mathcal{A}$ is a contraction. As a consequence, $\mathcal{B}^m$ could be used as a preconditioner in the PCG method (cf.~\cite{Xu1992}).
\section{Laplace--Beltrami equation and linear nodal elements}\label{seclinear}
In $\mathbb{R}^{d+1}$, let $\mathcal{M}$ be  a closed Lipschitz hypersurface (the boundary $\partial\mathcal{M}=\emptyset$). Naturally $\mathcal{M}$ is endowed with a metric, which is the pullback of the Euclidean metric in $\mathbb{R}^{d+1}$ via the embedding $\mathcal{M}\hookrightarrow\mathbb{R}^{d+1}$. The orientation of $\mathcal{M}$ is given by its unit normal vector $\nu$. Given a Lipschitz function $v$ on $\mathcal{M},$ the tangential gradient field $\nabla_\mathcal{M}v$ on $\mathcal{M}$ is defined as
\begin{equation*}
    \nabla_\mathcal{M}v:=\nabla\tilde{v}-(\nu\cdot\nabla\tilde{v})\nu, 
\end{equation*}
where $\nabla$ is the full gradient in $\mathbb{R}^{d+1},$ $\tilde{v}$ is an extension of $v$ in an neighborhood of $\mathcal{M}.$ It is well known that $\nabla_\mathcal{M}v$ is independent of the choice of such extension.
 Let $d\mu$ denote the surface measure on $\mathcal{M}$ and $L^2(\mathcal{M})$ the space of $L^2$ integrable functions on $\mathcal{M}$. We make use of the following $L^2(\mathcal{M})$ inner product
\begin{equation*}
    (v,w)_{\mathcal{M}}=\int_\mathcal{M}vwd\mu,
\end{equation*}
and the induced $L^2(\mathcal{M})$ norm $\|\cdot\|_\mathcal{M}$. The surface Sobolev space $H^1(\mathcal{M})$ is
\begin{equation*}
    V=H^1(\mathcal{M}):=\big\{v\in L^2(\mathcal{M}): \|\nabla_\mathcal{M} v\|_{\mathcal{M}}<\infty\big\}.
\end{equation*}
With the help of $(\cdot,\cdot)_\mathcal{M}$, the surface divergence $\text{div}_\mathcal{M}$ is defined as the $L^2(\mathcal{M})$ adjoint of $-\nabla_\mathcal{M}$. Then $\Delta_\mathcal{M}=\text{div}_\mathcal{M}\nabla_\mathcal{M}$ is the famous Laplace--Beltrami operator (surface Laplacian) on $\mathcal{M}$. 

Given $c\in\{0,1\}$ and $f\in L^2(\mathcal{M})$, we consider the second order elliptic equation
\begin{equation}\label{rd}
    -\Delta_\mathcal{M}u+cu=f\text{  on }\mathcal{M}.
\end{equation}
When $c=0$, \eqref{rd} reduces to the Laplace--Beltrami equation.
The variational formulation of \eqref{rd} seeks $u\in V$ such that
\begin{equation}\label{LB}
		(\nabla_\mathcal{M} u,\nabla_\mathcal{M} v)_\mathcal{M}+c(u,v)_\mathcal{M}=(f,v)_{\mathcal{M}},\quad     v\in V.
\end{equation}
For $c=0,$ the solution of \eqref{LB} is uniquely determined modulo a constant. We shall develop efficient iterative solvers for several popular discretizations of \eqref{LB}.

\subsection{Linear nodal element discretization}\label{subsecDLB}
When devising numerical schemes for solving \eqref{LB}, $\mathcal{M}$ is often approximated by a polyhedral (polygonal when $d=2$) hypersurface  $\mathcal{M}_h$ with simplicial $d$-dimensional faces. The definitions of $\nabla_\mathcal{M}$, $(\cdot,\cdot)_\mathcal{M}$, $\|\cdot\|_\mathcal{M}$ might be extended to any Lipschitz set $U$ (such as $\mathcal{M}_h$) and be denoted as $\nabla_U$, $(\cdot,\cdot)_U$, $\|\cdot\|_U$. On $\mathcal{M}_h$, the $L^2$ norm $\|\cdot\|_{\mathcal{M}_h}$ is simplified as $\|\cdot\|.$ Let  $\mathcal{T}_h$ be the collection of all $d$-faces of $\mathcal{M}_h.$  
By $\mathcal{P}_r(U)$ we denote the space of polynomials at most degree $r$ on a Lipschitz set $U.$ The conforming linear finite element space is
\begin{align*}
    V_h=\{v_h\in H^1(\mathcal{M}_h): v_h|_\tau\in\mathcal{P}_1(\tau)~\forall \tau\in\mathcal{T}_h\}.
\end{align*}
The finite element discretization of \eqref{LB} is to find $u_h\in V_h$ such that
\begin{equation}\label{dLB}
		(\nabla_{\mathcal{M}_h} u_h,\nabla_{\mathcal{M}_h} v_h)_{\mathcal{M}_h}+c(u_h,v_h)_{\mathcal{M}_h}=(f_h,v_h)_{\mathcal{M}_h},\quad     v_h\in V_h,
\end{equation}
where $f_h\in L^2(\mathcal{M}_h)$  approximates $f$. Let $d\mu_h$ denote the surface measure on $\mathcal{M}_h,$ and $|\mathcal{M}_h|:=\int_{\mathcal{M}_h}1d\mu_h.$ When $c=0$, we further require that \begin{equation*}
    \bar{f}_h:=\frac{1}{|\mathcal{M}_h|}\int_{\mathcal{M}_h}f_hd\mu_h=0
\end{equation*}
such that \eqref{dLB} is uniquely solvable. The bilinear form in \eqref{dLB} induces a linear operator $A_h=A_h^c: V_h\rightarrow V_h^\prime$ given by
\begin{equation*}
    \langle A^c_hv_h,w_h\rangle=(\nabla_{\mathcal{M}_h} v_h,\nabla_{\mathcal{M}_h} w_h)_{\mathcal{M}_h}+c(v_h,w_h)_{\mathcal{M}_h},\quad\forall v_h, w_h\in V_h.
\end{equation*}
Clearly $A_h$ is semi-definite when $c=0$, where the kernel $N(A^0_h)$ consists of constant functions on $\mathcal{M}_h.$
In the following, we present multilevel preconditioners for $A_h$ that provide uniformly bounded condition numbers. In doing so, it is assumed that 
\begin{itemize}
    \item[\textbf{A1.}] There exists a reference polyhedral hypersurface $\mathcal{M}_0$  and a bi-Lipschitz mapping $\Phi_h: \mathcal{M}_0\rightarrow\mathcal{M}_h$.
    \item[\textbf{A2.}] There exists a nested sequence of conforming and shape regular reference grids $\{\widehat{\mathcal{T}}_j\}_{j=0}^J$ on $\mathcal{M}_0$, where $\widehat{\mathcal{T}}_{j+1}$ is a refinement of $\widehat{\mathcal{T}}_j$ for each $j$, $\Phi_h(\widehat{\mathcal{T}}_J)=\mathcal{T}_h$, and the coarsest mesh $\widehat{\mathcal{T}}_0$ is the set of $d$-faces of $\mathcal{M}_0$. 
\end{itemize}

\begin{figure}[tbhp]
\centering
\subfloat[8 elements]{\label{Sa}\includegraphics[width=4.3cm]{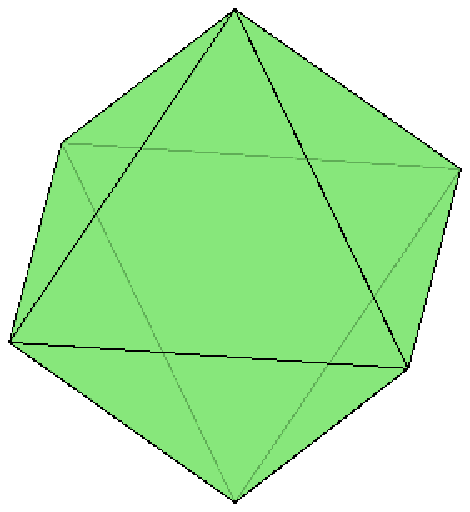}}
\subfloat[512 elements]{\label{Sb}\includegraphics[width=4.3cm]{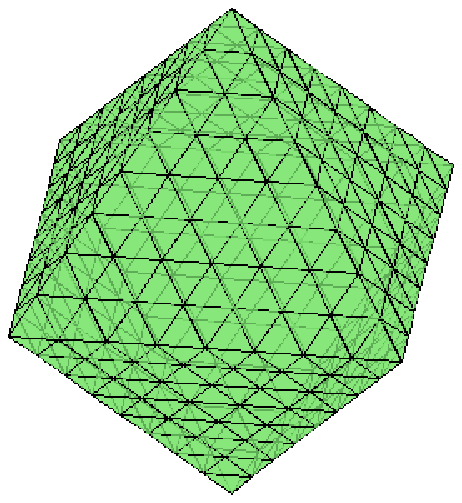}}
\subfloat[2048 elements]{\label{Sc}\includegraphics[width=4.3cm]{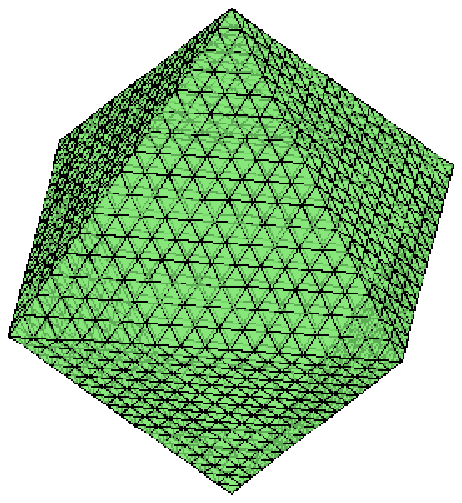}}
\caption{reference triangulations of a unit sphere $\mathbb{S}^2$ in $\mathbb{R}^3$}
\label{2sphere}
\end{figure}

\begin{figure}[tbhp]
\centering
\subfloat[8 elements]{\label{Sa1}\includegraphics[width=4.3cm]{sphere0.eps}}
\subfloat[512 elements]{\label{Sb1}\includegraphics[width=4.3cm]{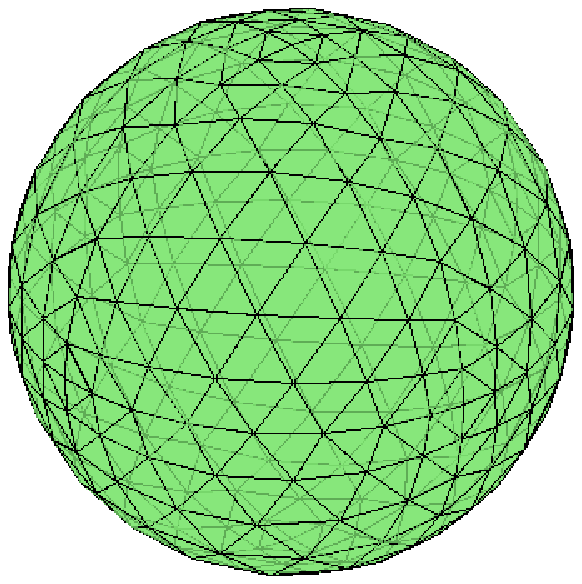}}
\subfloat[2048 elements]{\label{Sc1}\includegraphics[width=4.3cm]{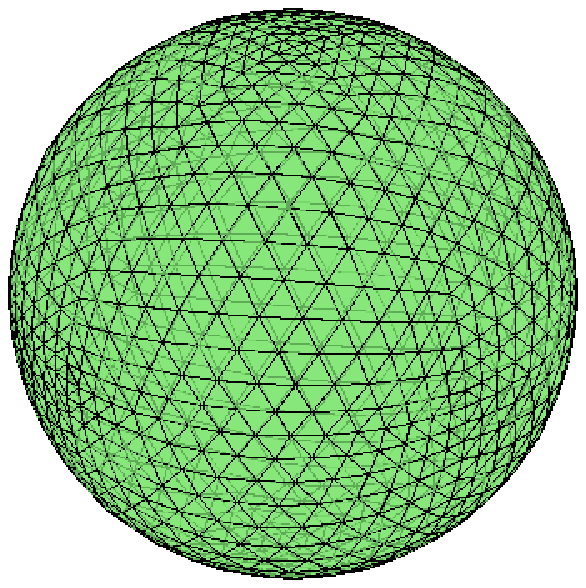}}
\caption{triangulations of a unit sphere $\mathbb{S}^2$ with vertices on $\mathbb{S}^2$}
\label{2sphere1}
\end{figure}

The grid hierarchy in \textbf{A2} is used to construct multilevel preconditioners for solving \eqref{dLB}.
Within a tubular neighborhood $\mathcal{U}\supset\mathcal{M}\cup\mathcal{M}_0,$ let $d(x)$ be the signed distance function of $\mathcal{M}$, that is, $|d(x)|=\text{dist}(x,\mathcal{M})$ $\forall x\in\mathcal{U}$. Then the unit normal $\nu$ of $\mathcal{M}$ is extended in $\mathcal{U}$ as $\nu(x)=\nabla d(x)/|\nabla d(x)|$ $\forall x\in\mathcal{U}$. The following function \begin{equation}\label{Phi}
    \Phi(x)=x-d(x)\nu(x)
\end{equation}
maps points in $\mathcal{U}$ onto $\mathcal{M}.$
In the classical literature, $\Phi$ is used to project grid vertices of the finest mesh $\widehat{\mathcal{T}}_J$ on $\mathcal{M}_0$ to $\mathcal{M}$ and construct the triangulated  surface $\mathcal{M}_h$, see, e.g., \cite{DemlowDziuk2007,BonitoCasconMekchayMorinNochetto2013,BonitoCasconMekchayMorinNochetto2016}. In this case, $\Phi_h$ in \textbf{A1} is determined by \begin{itemize}
    \item$\Phi_h(x)=\Phi(x)$ for each vertex $x$ in $\widehat{\mathcal{T}}_J$,
    \item$\Phi_h$ is linear on each element in $\widehat{\mathcal{T}}_J$.
\end{itemize}
For example, when $\mathcal{M}=\mathbb{S}^2=\{x\in\mathbb{R}^3: |x|=1\}$ is the unit sphere in $\mathbb{R}^3,$ the signed distance function is $d_{\mathbb{S}^2}=|x|-1$ and $\Phi=x/|x|$. A reference surface $\mathcal{M}_0$ for $\mathbb{S}^2$ and its partitions are shown in Fig.~\ref{2sphere}. The true meshes with vertices on $\mathbb{S}^2$ are shown in Fig.~\ref{2sphere1}. When $d(x)$ is not explicitly available, one could use local optimization processes in e.g., \cite{DemlowDziuk2007,DednerMadhavenStinner2013} to approximate $\Phi.$

We have the following transfer operator $\Pi_h: L^2(\mathcal{M}_0)\rightarrow L^2(\mathcal{M}_h)$ \begin{equation*}
    \Pi_hv=v\circ\Phi^{-1}_h.
\end{equation*}
Since $\Phi_h$ is bi-Lipschitz, $\Pi_h$ is a bijection that preserves $H^1$ regularity, i.e., \begin{equation*}
    \Pi_hH^1(\mathcal{M}_0)=H^1(\mathcal{M}_h).
\end{equation*}
The auxiliary finite element space on $\mathcal{M}_0$ is
\begin{equation*}
    \widehat{V}_h:=\Pi_h^{-1}V_h.
\end{equation*} 
In fact $\widehat{V}_h$ is just the space of globally continuous and piecewise linear polynomials on $\mathcal{M}_0$ with respect to the mesh $\widehat{\mathcal{T}}_h$. For any $v\in\widehat{V}_h$, $\hat{\tau}\in\widehat{\mathcal{T}}_h$ with $\tau=\Phi_h(\hat{\tau})$, it follows from
\textbf{A1} and \textbf{A2} that
\begin{subequations}\label{equiv}
\begin{align}
    &L_0^{-1}\|v\|_{\hat{\tau}}\leq\|\Pi_h v\|_{\tau}\leq L_0\|v\|_{\hat{\tau}},\\
    &L_1^{-1}|v|_{H^1(\hat{\tau})}\leq|\Pi_h v|_{H^1(\tau)}\leq L_1|v|_{H^1(\hat{\tau})},
\end{align}
\end{subequations}
where $L_0>0$, $L_1>0$ rely on the Lipschitz constant of $\Phi_h$. In \cite{BonitoCasconMekchayMorinNochetto2016}, the nodal interpolant $\Phi_h$ is shown to be uniformly bi-Lipschitz provided the mesh size $\widehat{\mathcal{T}}_J$ is sufficiently small. In this case, $L_0$, $L_1$ are absolute constants. Throughout the rest of this paper, we say $C_1\lesssim C_2$ provided  $C_1\leq CC_2$ with $C$ a generic constant depending only on $L_0, L_1,$ $\mathcal{M}_0,$ $\mathcal{M},$ and the shape-regularity of $\widehat{\mathcal{T}}_h.$ We say $C_1\approx C_2$ provided  $C_1\lesssim C_2$ and $C_2\lesssim C_1$.

\subsection{Preconditioners for linear nodal elements}\label{subseclinear}

We consider the linear operator
$\widehat{A}_h: \widehat{V}_h\rightarrow \widehat{V}_h^\prime$ determined by
\begin{equation*}
    \langle \widehat{A}_hv_h,w_h\rangle=(\nabla_{\mathcal{M}_0}v_h,\nabla_{\mathcal{M}_0} w_h)_{\mathcal{M}_0}+(v_h,w_h)_{\mathcal{M}_0},\quad v_h, w_h\in\widehat{V}_h.
\end{equation*}
It is noted that the auxiliary operator $\widehat{A}_h$ is always SPD,
although $A_h$ with $c=0$ is singular. The next theorem is the first main result in this paper.
\begin{theorem}\label{thm1}
Let $\widehat{B}_h: \widehat{V}^\prime_h\rightarrow\widehat{V}_h$ be a preconditioner for $\widehat{A}_h.$ Then the operator
$B_h:=\Pi_h\widehat{B}_h\Pi_h^\prime$ satisfies
\begin{equation*}
    \kappa(B_hA_h)\lesssim\kappa(\widehat{B}_h\widehat{A}_h).
\end{equation*} 
\end{theorem}
\begin{proof}
It suffices to check the three assumptions in Corollary \ref{FSP2} with $\mathcal{V}=V_h$, $\widetilde{\mathcal{V}}=\widehat{V}_h$, $\widetilde{\mathcal{B}}=\widehat{B}_h$, $\mathcal{A}=A_h$, $\widetilde{\mathcal{A}}=\widehat{A}_h$. The first assumption is verified by $N(\widehat{A}_h)=\{0\}$ and $\Pi_h(N(\widehat{A}_h))\subseteq N(A_h).$ The second assumption follows from \eqref{equiv}. Given $v\in V_h$, we define
\begin{equation*}
   \tilde{v}= \left\{\begin{aligned}
    &\Pi_h^{-1}v,\quad \text{ if }c=1,\\
    &\Pi_h^{-1}v-\overline{\Pi_h^{-1}v},\quad\text{ if } c=0,
\end{aligned}\right.
\end{equation*}
where $\overline{\Pi_h^{-1}v}$ is the average of $\Pi_h^{-1}v$ on $\mathcal{M}_h$. In either case, we have $\Pi_h\tilde{v}-v\in N(A_h)$ and  $\|\tilde{v}\|_{\widehat{A}_h}\lesssim|v|_{A_h}$ by \eqref{equiv} (and the Poincar\'e inequality on $\mathcal{M}_0$ when $c=0$). Therefore the third assumption in Corollary \ref{FSP2} is confirmed. 
\end{proof}

In view of Theorem \ref{thm1}, it remains to construct a multilevel preconditioner $\widehat{B}_h$ for $\widehat{A}_h$ on the reference hypersurface $\mathcal{M}_0$. To that end, we consider the nested spaces 
\begin{equation}\label{nestedV}
    \widehat{V}_0\subset \widehat{V}_1\subset \widehat{V}_2\subset\cdots\subset\widehat{V}_J=\widehat{V}_h,
\end{equation}
where $\widehat{V}_j$ is the space of continuous and piecewise linear polynomials with respect to $\widehat{\mathcal{T}}_j$ in \textbf{A1}. Recall that $\widehat{A}_h$ is a SPD operator from conforming linear nodal elements on a fixed and piecewise flat surface $\mathcal{M}_0$. Therefore optimal additive and multiplicative preconditioners $\widehat{B}_h$ of $\widehat{A}_h$ directly follow from \eqref{nestedV} and the classical geometric multilevel theory, see, e.g., \cite{GriebelOswald1995,Hackbusch1985,Xu1992,XuZikatanov2002}. To be precise, we next briefly describe the construction of the multilevel preconditioner $\widehat{B}_h$ on $\mathcal{M}_0$ in the framework of subspace correction methods (cf.~\cite{Xu1992,XuZikatanov2002}). Let $\widehat{V}_h$ be decomposed as 
\begin{equation}\label{VhVk}
    \widehat{V}_h=\sum_{\lambda\in\Lambda}\widehat{V}_\lambda,
\end{equation}
where $\{\widehat{V}_\lambda\}$ are subspaces of $\widehat{V}_h$ and $\Lambda=\{\lambda_k\}_{k=1}^K$ is a finite index set.
Let $I_\lambda: \widehat{V}_\lambda\rightarrow\widehat{V}_h$ be the inclusion, $\widehat{Q}_\lambda=\widehat{I}_\lambda^\prime: \widehat{V}^\prime_\lambda\rightarrow\widehat{V}^\prime_h$, and $\widehat{A}_\lambda=\widehat{Q}_\lambda \widehat{A}_h\widehat{I}_\lambda$. Given a fixed initial guess in $\widehat{V}_h$, the multiplicative preconditioner $\widehat{B}^m_h: \widehat{V}_h^\prime\rightarrow \widehat{V}_h$  for $\widehat{A}_h$ is a successive subspace correction method given in Algorithm \ref{wowo}.

\begin{algorithm}[H]
\caption{Symmetrized Successive Subspace Correction $\widehat{B}_h^m$}\label{wowo}

Input: $u_0\in\widehat{V}_h$ and $g\in\widehat{V}^\prime_h$.

\textbf{For} {$k=1:K$}

$\qquad u_k=u_{k-1}+\widehat{A}_{\lambda_k}^{-1}\widehat{Q}_{\lambda_k}(g-\widehat{A}_hu_{k-1})$;

\textbf{EndFor}

\textbf{For} {$k=K+1:2K$}

$\qquad u_k=u_{k-1}+\widehat{A}_{\lambda_{2K+1-k}}^{-1}\widehat{Q}_{\lambda_{2K+1-k}}(g-\widehat{A}_hu_{k-1})$;

\textbf{EndFor}

Output: $\widehat{B}_h^mg:=u_{2K}.$
\end{algorithm}
Based on the decomposition \eqref{VhVk}, the additive preconditioner $\widehat{B}_h^a$ for $\widehat{A}_h$ reads
\begin{equation}\label{Bhahat}
    \widehat{B}_h^a=\sum_{\lambda\in\Lambda}\widehat{A}_\lambda^{-1}\widehat{Q}_\lambda.
\end{equation}

For a \emph{quasi-uniform} mesh sequence $\{\widehat{\mathcal{T}}_j\}_{j=0}^J$, each element $\hat{\tau}$ in $\widehat{\mathcal{T}}_j$ satisfies $\text{diam}(\hat{\tau})$ $\approx2^{-j}$, $0\leq j\leq J$. Let $\{z_{j,l}\}_{l=1}^{n_j}$ denote the set of vertices in $\widehat{\mathcal{T}}_j,$ $\phi_{j,l}$ the hat basis function of $\widehat{V}_j$ at $z_{j,l}$, and $\widehat{V}_{j,l}=\text{span}\{\phi_{j,l}\}.$ Based on the multilevel subspace decomposition  \begin{equation}\label{mguniform}
    \widehat{V}_h=\sum_{\lambda\in\Lambda}\widehat{V}_\lambda:=\sum_{j=1}^J\sum_{l=1}^{n_j}\widehat{V}_{j,l}+\widehat{V}_0,
\end{equation}
we obtain multilevel preconditioners $\widehat{B}_h^m$ and $\widehat{B}_h^a$ for $\widehat{A}_h$ given in Algorithm \ref{wowo} and \eqref{Bhahat}.
The abstract theory in \cite{XuZikatanov2002} implies
\begin{equation}\label{Bhambound}
    \kappa(\widehat{B}_h^a\widehat{A}_h)\lesssim1,\quad \kappa(\widehat{B}_h^m\widehat{A}_h)\lesssim1,
\end{equation}
see also \cite{ChenNochettoXu2012} for details. From an algorithmic viewpoint, $\widehat{B}_h^m$ corresponds to the multigrid V-cycle with Gauss--Seidel smoothers, and $\widehat{B}_h^a$ leads to a BPX-type multilevel preconditioner. 

\begin{remark}
In the above analysis, the initial mesh size of $\widehat{\mathcal{T}}_0$ is not required to be sufficiently small. Therefore the proposed  preconditioners avoid the assumption on the smallness of the coarsest grid size used in \cite{BonitoPasciak2012}.
\end{remark}

For surface FEMs driven by a posteriori error estimators (cf.~\cite{DemlowDziuk2007,BonitoCasconMekchayMorinNochetto2016}), the adaptively generated sequence $\{\widehat{\mathcal{T}}_j\}_{j=0}^J$ is not quasi-uniform in general. In this case, the decomposition \eqref{mguniform} yields non-optimal preconditioners due to inefficient smoothing at each level. To maintain optimal complexity, one could use local multilevel methods, smoothing only at newly added vertices and part of their neighbors on each level, see, e.g.,  \cite{WuChen2006,GrasedyckWangXu2016,Yserentant1993,DahmenKunoth1992,AksoyluHolst2006}. Alternatively, the work \cite{ChenNochettoXu2012} uses a special mesh coarsening algorithm to construct nested grid sequences and optimal multilevel preconditioners with provable condition number bound \eqref{Bhambound} on graded bisection grids.

\section{Nonconforming and DG methods}\label{secNCDG}
In this section, we develop preconditioners for nonconforming linear  and DG discretizations of \eqref{LB}. 
The DG space on $\mathcal{M}_h$ is defined as
\begin{align*}
    \mathbb{V}_h:=\big\{v_h\in L^2(\mathcal{M}_h): v_h|_\tau\in\mathcal{P}_1(\tau)~\forall \tau\in\mathcal{T}_h\big\}.
\end{align*}
Let $\mathcal{E}_h$ denote the set of $(d-1)$-dimensional simplexes in $\mathcal{T}_h$, and $\mathcal{N}_h$ be the set of vertices in $\mathcal{T}_h$. 
For each $e\in\mathcal{E}_h$, let $\tau_e^+, \tau_e^-\in\mathcal{T}_h$ be the two elements sharing $e$, and $\omega_e=\tau_e^+\cup\tau_e^-$. The outward unit \emph{conormal} $\nu_e^+$ (resp.~$\nu_e^+$) to $e\subset\partial\tau_e^+$ (resp.~$e\subset\partial\tau_e^-$) is a unit vector parallel to $\tau_e^+$ (resp.~$\tau_e^-$) and normal to $e$. By $\nabla_h$ we denote the broken surface gradient on $\mathcal{M}_h$ such that $(\nabla_hv_h)|_\tau=\nabla_\tau(v_h|_\tau)$ $\forall\tau\in\mathcal{T}_h.$ For $v_h, w_h\in\mathbb{V}_h$, let
\begin{align*}
    \llbracket v_h\rrbracket|_e&:=v_h|_{\tau_e^+}-v_h|_{\tau_e^-},\\
    \{\nabla_hv_h\cdot\nu_h\}|_e&:=\frac{1}{2}\big(\nabla_hv_h|_{\tau_e^+}\cdot\nu_e^+-\nabla_hv_h|_{\tau_e^-}\cdot\nu_e^-\big),\\
    (\nabla_hv_h,\nabla_hw_h)_{\mathcal{M}_h}&:=\sum_{\tau\in\mathcal{T}_h}(\nabla_\tau v_h,\nabla_\tau w_h)_\tau.
\end{align*}
Given a subset $\mathcal{E}\subseteq\mathcal{E}_h$, we define
\begin{equation*}
    \langle \xi,\eta\rangle_\mathcal{E}:=\sum_{e\in\mathcal{E}}\int_e\xi\eta d\mu_e,
\end{equation*}
with $d\mu_e$ the $(d-1)$-Lebesgue measure on $e.$ 
Let $\|\cdot\|_\mathcal{E}$ denote the $L^2$ norm induced by $\langle \cdot,\cdot\rangle_\mathcal{E},$ $Q_h$ the $L^2$ projection onto $\prod_{e\in\mathcal{E}_h}\mathcal{P}_0(e),$ and $h_\mathcal{T},$ $h_\mathcal{E}$ the mesh size functions such that \begin{align*}
    &h_\mathcal{T}|_\tau=h_\tau=\text{diam}(\tau),\quad\forall\tau\in\mathcal{T}_h,\\
    &h_\mathcal{E}|_e=h_e=\text{diam}(e),\quad\forall e\in\mathcal{E}_h.  
\end{align*}
We make use of the following bilinear forms $a_h^{\text{CR}}$, $a_h^{\text{DG}}$ and linear operators $\mathcal{A}_h=\mathcal{A}_h^c: \mathcal{V}_h\rightarrow\mathcal{V}_h^\prime$, $\mathbb{A}_h=\mathbb{A}_h^c: \mathbb{V}_h\rightarrow\mathbb{V}_h^\prime$
\begin{align*}
    &a_h^{\text{CR}}(v_h,w_h)=\langle\mathcal{A}^c_hv_h,w_h\rangle=(\nabla_h v_h,\nabla_h w_h)_{\mathcal{M}_h}+c(v_h,w_h),\\
    &a_h^{\text{DG}}(v_h,w_h)=\langle\mathbb{A}^c_hv_h,w_h\rangle=(\nabla_h v_h,\nabla_h w_h)_{\mathcal{M}_h}-\big\langle\{\nabla_hv_h\cdot\nu_h\},\lr{w_h}\big\rangle_{\mathcal{E}_h}\\
    &\quad-\big\langle\{\nabla_hw_h\cdot\nu_h\},\lr{v_h}\big\rangle_{\mathcal{E}_h}+\alpha\big\langle h_\mathcal{E}^{-1}Q_h\lr{v_h},Q_h\lr{w_h}\big\rangle_{\mathcal{E}_h}+c(v_h,w_h),
\end{align*}
where $\alpha>0$ is a constant.

The nonconforming CR element method for \eqref{LB} seeks $u^{\text{CR}}_h\in\mathcal{V}_h$ such that
\begin{equation}\label{CRLB}
		a_h^{\text{CR}}(u^{\text{CR}}_h,v_h)=(f_h,v_h)_{\mathcal{M}_h},\quad     v_h\in \mathcal{V}_h.
\end{equation}
Here the CR element space $\mathcal{V}_h$ on $\mathcal{M}_h$ is given by
\begin{align*}
    \mathcal{V}_h:=\big\{v_h\in \mathbb{V}_h: Q_h\lr{v_h}=0\big\}.
\end{align*}
For \eqref{CRLB}, the work \cite{Guo2020} derives a priori error estimate and superconvergent gradient recovery technique.
The DG method for \eqref{LB} is to find $u^{\text{dG}}_h\in\mathbb{V}_h$
\begin{equation}\label{DGLB}
		a_h^{\text{DG}}(u^{\text{DG}}_h,v_h)=(f_h,v_h)_{\mathcal{M}_h},\quad     v_h\in \mathbb{V}_h.
\end{equation}
Assuming that $\alpha$ is sufficiently large, \eqref{DGLB} is proposed and analyzed in \cite{DednerMadhavenStinner2013,AntoniettiDednerMadhavenStangalinoStinnerBjornVerani2015}.

In some situations, the global assumptions \textbf{A1} and \textbf{A2} might be demanding. Now we propose several local assumptions on local mesh quality of $\mathcal{T}_h$ as follows.
\begin{itemize}
    \item[\textbf{A3.}] The triangulation $\mathcal{T}_h$ is shape regular, i.e., there exists an absolute constant $\gamma_0>0$ such that $r_T/\rho_T\leq\gamma_0$ $\forall T\in\mathcal{T}_h$, where $r_T$, $\rho_T$ are radii of circumscribed and inscribed spheres of $T.$
    \item[\textbf{A4.}] For each vertex $z\in\mathcal{N}_h$, let $N_z$ denote the number of elements sharing $z$ in $\mathcal{T}_h$. There exists an absolute integer $N_0$ such that $N_z\leq N_0$ for all $z$.
    \item[\textbf{A5.}] Let $\hat{\omega}=\hat{\tau}^+\cup\hat{\tau}^-$ be the union of two simplexes $\hat{\tau}^+$, $\hat{\tau}^-$ sharing a $(d-1)$-dimensional face in $\mathbb{R}^d$. There exists an absolute constant $L>0$ and a parametrization $\varphi_e: \hat{\omega}\rightarrow\mathbb{R}^d$ of $\omega_e$ such that
\begin{align*}
    &\varphi_e(\hat{\omega})=\omega_e,\quad\varphi_e|_{\tau^+},~\varphi_e|_{\tau^-}\text{ are affine},\\
	&L^{-1}|\hat{x}-\hat{y}|\leq|\varphi_e(\hat{x})-\varphi_e(\hat{y})|\leq L|\hat{x}-\hat{y}|,\quad\forall\hat{x}, \hat{y}\in\hat{\omega}.
\end{align*}
\end{itemize}
Here \textbf{A3}--\textbf{A5} are local and weaker than \textbf{A1} and \textbf{A2}. 

In addition, the following Poincar\'e inequality 
\begin{equation}\label{Poincare}
    \|v-\bar{v}\|\leq c_P\|\nabla_{\mathcal{M}_h}v\|,\quad\forall v\in H^1(\mathcal{M}_h)
\end{equation}
is useful in the semi-definite case $c=0,$ where $\bar{v}$ is the average of $v$ on $\mathcal{M}_h.$
We say $C_1\preccurlyeq C_2$ provided  $C_1\leq CC_2$ with $C$ being a generic constant depending only on $L$, $N_0$, $\gamma_0,$ $c_P$. We say $C_1\simeq C_2$ provided  $C_1\preccurlyeq C_2$ and $C_2\preccurlyeq C_1$.

For $z\in\mathcal{N}_h$, let $\omega_z$ be the union of elements in $\mathcal{T}_h$ sharing $z$ as a vertex. 
To derive auxiliary space preconditioners for $\mathcal{A}_h$ and $\mathbb{A}_h$, we need the following nodal averaging process $I^{\text{av}}_h: \mathbb{V}_h\rightarrow V_h$ given by
\begin{equation*}
    (I^{\text{av}}_hv_h)(z):=\frac{1}{N_z}\sum _{\tau\in\mathcal{T}_h,\tau\subset\omega_z}v_h|_\tau(z),\quad\forall v_h\in\mathbb{V}_h,~\forall z\in\mathcal{N}_h.
\end{equation*}
The next lemma discusses the stability and approximation property of $I^{\text{av}}_h$.
\begin{lemma}\label{approxI}
For any $v_h\in\mathbb{V}_h,$ it holds that
\begin{subequations}
\begin{align}
    \|I^{\emph{av}}_hv_h\|&\preccurlyeq\|v_h\|,\\
    \|\nabla_{\mathcal{M}_h}I^{\emph{av}}_hv_h\|&\preccurlyeq\|\nabla_hv_h\|+\|h_\mathcal{E}^{-\frac{1}{2}}Q_h\lr{v_h}\|_{\mathcal{E}_h},\label{Ih1}\\
    \|h_\mathcal{T}^{-1}(v_h-I^{\emph{av}}_hv_h)\|&\preccurlyeq\|\nabla_hv_h\|+\|h_\mathcal{E}^{-\frac{1}{2}}Q_h\lr{v_h}\|_{\mathcal{E}_h}.\label{Ih2}
\end{align}
\end{subequations}
\end{lemma}
\begin{proof}
For each $\tau\in\mathcal{T}_h$, let $\mathcal{N}_\tau$ be the set of vertices of $\tau.$ For $z\in\mathcal{N}_\tau$,
\begin{equation}\label{jump1}
    (v_h-I^{\text{av}}_hv_h)(z)=\frac{1}{N_z}\sum _{\tilde{\tau}\in\mathcal{T}_h,\tilde{\tau}\subset\omega_z}\big(v_h|_\tau(z)-v_h|_{\tilde{\tau}}(z)\big).
\end{equation}
If $\tilde{\tau}\in\mathcal{T}_h$ is an element in $\omega_z$ that shares a $(d-1)$-simplex $e$ with $\tau,$ then using the parametrization $\varphi_e$ of $\omega_e$ in \textbf{A5} and a scaling argument, we  obtain
\begin{equation}\label{vtautauprime}
    \big|v_h|_\tau(z)-v_h|_{\tilde{\tau}}(z)\big|\preccurlyeq h_e^{1-\frac{d}{2}}\|\nabla_hv_h\|_{\omega_e}+ h_e^{\frac{1}{2}-\frac{d}{2}}\|Q_h\lr{v_h}\|_e.
\end{equation}
In general, for any $\tilde{\tau}\in\mathcal{T}_h$ contained in $\omega_z$, we still arrive at \eqref{vtautauprime} by applying \eqref{vtautauprime} to a chain of pairs of adjacent simplexes in $\omega_z$,  see, e.g., \cite{KarakashianPascal2003}. Let $\mathcal{E}_\tau\subset\mathcal{E}_h$ denote the set of $(d-1)$-simplexes containing at least one vertex of $\tau,$ and $\omega_\tau$ the union of elements sharing at least one vertex with $\tau$ in $\mathcal{T}_h.$
It then follows from finite-dimensional norm equivalence, \textbf{A3}, and \eqref{jump1}, \eqref{vtautauprime} that
\begin{equation}\label{local1}
h_\tau^{-2}\|v_h-I^{\text{av}}_hv_h\|^2_\tau\simeq h_\tau^{d-2}\sum_{z\in\mathcal{N}_\tau}(v_h-I^{\text{av}}_hv_h)(z)^2\preccurlyeq\|\nabla_hv_h\|^2_{\omega_\tau}+\|h_\mathcal{E}^{-1}Q_h\lr{v_h}\|_{\mathcal{E}_\tau}^2.
\end{equation}
Similar local argument leads to 
\begin{equation}\label{local2}
\begin{aligned}
    \|\nabla_{\mathcal{M}_h}I^{\text{av}}_hv_h\|^2_\tau&\preccurlyeq\|\nabla_hv_h\|^2_{\omega_\tau}+\|h_\mathcal{E}^{-1}Q_h\lr{v_h}\|_{\mathcal{E}_\tau}^2,\\
    \|I^{\text{av}}_hv_h\|^2_\tau&\preccurlyeq\|v_h\|^2_{\omega_\tau}.
    \end{aligned}
\end{equation}
Summing \eqref{local1}, \eqref{local2} over $\tau\in\mathcal{T}_h$ and using \textbf{A4} complete the proof.
\end{proof}
Combining Lemma \ref{approxI} and \eqref{Poincare} with the triangle inequality
\begin{equation*}
    \|v_h-\bar{v}_h\|\leq\|v_h-I_h^{\text{av}}v_h\|+\|I_h^{\text{av}}v_h-\overline{I_h^{\text{av}}v_h}\|+\|\overline{I_h^{\text{av}}v_h-v_h}\|,
\end{equation*}
we obtain a discrete Poincar\'e inequality
\begin{equation}\label{disPoincare}
    \|v_h-\bar{v}_h\|\preccurlyeq\|\nabla_hv_h\|+\|h_\mathcal{E}^{-1}Q_h\lr{v_h}\|_{\mathcal{E}_h},\quad\forall v_h\in\mathbb{V}_h
\end{equation}

\subsection{Preconditioners for the CR method}\label{subsecCR}
Now we are in a position to present the auxiliary space preconditioner for the CR element method \eqref{CRLB}. 
\begin{theorem}\label{CRpreconditioner}
Let $\mathcal{S}_h: \mathcal{V}_h^\prime\rightarrow\mathcal{V}_h$ be SPD and satisfy 
\begin{equation}\label{ShCR}
      \|v_h\|_{\mathcal{S}_h^{-1}}\simeq\|h_\mathcal{T}^{-1}v_h\|,\quad\forall v_h\in \mathcal{V}_h.
\end{equation}
Let $\mathcal{I}_h: V_h\rightarrow\mathcal{V}_h$ be the inclusion,  $\mathcal{Q}_h=\mathcal{I}_h^\prime,$ and $B_h: V^\prime_h\rightarrow V_h$  a preconditioner for $A^1_h.$ Then $\mathcal{B}^a_h=\mathcal{S}_h+B_h\mathcal{Q}_h: \mathcal{V}_h^\prime\rightarrow\mathcal{V}_h$ 
is a preconditioner for $\mathcal{A}_h$ such that
\begin{equation*}
    \kappa(\mathcal{B}^a_h\mathcal{A}_h)\preccurlyeq\kappa(B_hA^1_h).
\end{equation*}
\end{theorem}
\begin{proof}
For $v_h\in\mathcal{V}_h$, it follows that $Q_h\lr{v_h}=0$. When $c=1$, we could use \eqref{Ih1}, \eqref{Ih2}, \eqref{ShCR}, \eqref{Poincare}  to verify those assumptions 
in Lemma \ref{ASP} with $\mathcal{V}=\mathcal{V}_h$, $V=V_h$,  $\mathcal{A}=\mathcal{A}^1_h$, $A=A_h^1$, $B=B_h$, $P=I_h^{\text{av}}$, and obtain that
\begin{equation}\label{curlyAh1}
    \kappa(\mathcal{B}^a_h\mathcal{A}_h^1)\lesssim\kappa(B_hA_h^1),
\end{equation}
which completes the proof if $c=1$. When $c=0,$
it follows from Lemma \ref{FSP2} with $\tilde{\mathcal{V}}=\mathcal{V}=\mathcal{V}_h$, $\mathcal{A}=\mathcal{A}_h$, $\tilde{\mathcal{A}}=\mathcal{A}^1_h$, $\tilde{\mathcal{B}}=\mathcal{B}^a_h$, $\Pi=I_{\mathcal{V}_h}$ and  \eqref{disPoincare}  that 
\begin{equation}\label{curlyAh}
    \kappa(\mathcal{B}^a_h\mathcal{A}_h)\preccurlyeq\kappa(\mathcal{B}^a_h\mathcal{A}_h^1).
\end{equation}
Combining \eqref{curlyAh1} with \eqref{curlyAh} completes the proof.
\end{proof}

Let $\{\phi_i\}_{1\leq i\leq n}$ be a canonical basis of $\mathcal{V}_h$ and $\{\phi^\prime_j\}_{1\leq j\leq n}$ the dual basis of $\mathcal{V}_h^\prime$ such that $\langle\phi^\prime_i,\phi_j\rangle=1$ if $i=j$ and $\langle\phi^\prime_i,\phi_j\rangle=0$ otherwise. Under these basis, $\mathcal{A}^1_h$ is represented as a matrix $\widetilde{\mathcal{A}}^1_h=(a_{ij})_{1\leq i,j\leq n}$ with $a_{ij}=\langle \mathcal{A}_h^1\phi_i,\phi_j\rangle$. The classical Jacobi relaxation $\mathcal{S}_h^a: \mathcal{V}_h^\prime\rightarrow\mathcal{V}_h$ of $\mathcal{A}_h$ is given as 
\begin{equation*}
    \mathcal{S}_h^av^\prime=\sum_ia_{ii}^{-1}\langle v^\prime,\phi_i\rangle\phi_i,\quad v^\prime\in\mathcal{V}_h^\prime.
\end{equation*}
Then the representing matrix of $\mathcal{S}_h^a$ is the diagonal matrix $\text{diag}(a_{11}^{-1},\ldots,a_{nn}^{-1})$.
It is straightforward to check that the smoother $\mathcal{S}_h=\mathcal{S}_h^a$ fulfills the assumption \eqref{ShCR} in Theorem \ref{CRpreconditioner}, namely, 
\begin{equation}\label{Sha}
      \|v_h\|_{(\mathcal{S}_h^a)^{-1}}\simeq\|h_\mathcal{T}^{-1}v_h\|,\quad\forall v_h\in \mathcal{V}_h.
\end{equation}
The preconditioner $B_h$ on the conforming subspace $V_h$ could be the geometric multigrid analyzed in Subsection \ref{subseclinear} or simply AMG.

Let $\mathcal{D}_h$, $\mathcal{L}_h$ be the diagonal and strict lower triangular part of  $\widetilde{\mathcal{A}}^1_h$, respectively. Let $\mathcal{S}_h^m: \mathcal{V}_h^\prime\rightarrow\mathcal{V}_h$ be the forward Gauss--Seidel relaxation of $\mathcal{A}^1_h$, which corresponds to the matrix $(\mathcal{D}_h+\mathcal{L}_h)^{-1}$. We obtain a two-level multiplicative preconditioner for the CR method. 
\begin{theorem}\label{CRpreconditioner2}
Let $\mathcal{B}_h^m$ be the two-level method in Algorithm \ref{twolevelSSC} with $\mathcal{V}=\mathcal{V}_h$, $V=V_h,$ $\mathcal{A}=\mathcal{A}^1_h,$ $\mathcal{S}=\mathcal{S}^m_h$, $A=A_h^1$. Then $\mathcal{B}_h^m$ is a preconditioner for $\mathcal{A}_h$ satisfying
\begin{equation*}
    \kappa(\mathcal{B}_h^m\mathcal{A}_h)\preccurlyeq1.
\end{equation*}
\end{theorem}
\begin{proof}
Given $v_h\in\mathcal{V}_h$, it is shown in \cite{Zikatanov2008,XuZikatanov2017} that 
\begin{equation}\label{JG}
    \|v_h\|_{(\bar{\mathcal{S}}_h^m)^{-1}}\simeq \|v_h\|_{(\mathcal{S}_h^a)^{-1}}.
\end{equation}
It follows from \eqref{JG}, \eqref{Sha}, and \eqref{Ih2} that
\begin{equation}\label{Sham}
\begin{aligned}
    &\inf_{w\in \mathcal{V}_h}\|v_h-w\|_{(\bar{\mathcal{S}}_h^m)^{-1}}\simeq\inf_{w\in \mathcal{V}_h}\|v_h-w\|_{(\mathcal{S}_h^a)^{-1}}\\
    &\simeq\inf_{w\in \mathcal{V}_h}\|h^{-1}_\mathcal{T}(v_h-w)\|\leq\|h^{-1}_\mathcal{T}(v_h-I_h^{\text{av}}v_h)\|\preccurlyeq\|\nabla_hv_h\|.
\end{aligned}
\end{equation}
Therefore combining \eqref{Sham} and Lemma \ref{ASP2} shows that $I-\mathcal{B}_h^m\mathcal{A}_h^1$ is a contraction and $\kappa(\mathcal{B}_h^m\mathcal{A}_h^1)\preccurlyeq1$. When $c=0$, we use the same analysis in the proof of Theorem \ref{CRpreconditioner} to obtain $\kappa(\mathcal{B}_h^m\mathcal{A}_h)\preccurlyeq\kappa(\mathcal{B}_h^m\mathcal{A}_h^1)$.
\end{proof}
In practice, the exact inverse $A^{-1}=(A_h^{1})^{-1}$ in Algorithm \ref{twolevelSSC} could be replaced with geometric multigrid or AMG V-cycle for linear nodal elements.

\subsection{Preconditioners for the DG method}\label{subsecdG}
To guarantee the well-posedness of \eqref{DGLB}, the parameter $\alpha$ in $a_h^{\text{DG}}$ is required to be sufficiently large, see, e.g., \cite{DednerMadhavenStinner2013}. 
\begin{lemma}\label{dGcoercive}
There exists a constant $\alpha_0>0$ dependent on $\gamma_0, N_0,$ such that 
\begin{equation*}
  a_h^{\emph{DG}}(v_h,v_h)\simeq\|\nabla_hv_h\|^2+c\|v_h\|^2+\|h_\mathcal{E}^{-\frac{1}{2}}Q_h[v_h]\|_{\mathcal{E}_h}^2,\quad\forall v_h\in\mathbb{V}_h,
\end{equation*}
whenever $\alpha\geq\alpha_0$.
\end{lemma}

Lemma \ref{dGcoercive} verifies the positive definiteness of $\mathbb{A}_h.$
The next theorem presents a two-level additive preconditioner for $\mathbb{A}_h.$ The proof is identical to Theorem \ref{CRpreconditioner}.
\begin{theorem}\label{DGpreconditioner}
Let $\mathbb{S}_h: \mathbb{V}_h^\prime\rightarrow\mathbb{V}_h$ satisfy 
\begin{equation}\label{ShdG}
      \|v_h\|_{\mathbb{S}_h^{-1}}\simeq\|h_\mathcal{T}^{-1}v_h\|,\quad\forall v_h\in \mathbb{V}_h.
\end{equation}
Let $\mathbb{I}_h: V_h\rightarrow\mathbb{V}_h$ be the inclusion, $\mathbb{Q}_h=\mathbb{I}_h^\prime,$ and $B_h: V^\prime_h\rightarrow V_h$  a preconditioner for $A^1_h.$ When $\alpha\geq\alpha_0,$ $\mathbb{B}^a_h=\mathbb{S}_h+B_h\mathbb{Q}_h: \mathbb{V}_h^\prime\rightarrow\mathbb{V}_h$ is a preconditioner for $\mathbb{A}_h$ satisfying
\begin{equation*}
    \kappa(\mathbb{B}^a_h\mathbb{A}_h)\preccurlyeq\kappa(B_hA^1_h).
\end{equation*}
\end{theorem}

Let $\{\psi_i\}_{1\leq i\leq l}$ be a canonical basis of $\mathbb{V}_h$ and $\{\psi^\prime_j\}_{1\leq j\leq l}$ the dual basis of $\mathbb{V}_h^\prime$. Under these basis, $\mathbb{A}^1_h$ is realized as a matrix $\widetilde{\mathbb{A}}^1_h=(\langle\mathbb{A}_h^1\psi_i,\psi_j\rangle)_{1\leq i,j\leq l}$. The Jacobi relaxation $\mathbb{S}_h^a: \mathbb{V}_h^\prime\rightarrow\mathbb{V}_h$ of $\mathbb{A}^1_h$ is represented by the inverse of the diagonal of $\widetilde{\mathbb{A}}^1_h$. Similarly to $\mathcal{S}_h^a,$ the smoother $\mathbb{S}_h=\mathbb{S}_h^a$ satisfies the assumption \eqref{ShdG} in Theorem \ref{DGpreconditioner}.

Let $\mathbb{S}_h^m: \mathbb{V}_h^\prime\rightarrow\mathbb{V}_h$ be the forward Gauss--Seidel relaxation of $\mathbb{A}^1_h$. We also obtain a two-level multiplicative preconditioner for the DG method in the next theorem. The proof is identical to Theorem \ref{CRpreconditioner2}. 
\begin{theorem}\label{DGpreconditioner2}
Let $\mathbb{B}_h^m$ be the two-level method in Algorithm \ref{twolevelSSC} with $\mathcal{V}=\mathbb{V}_h$, $\mathcal{A}=\mathbb{A}^1_h,$ $\mathcal{S}=\mathbb{S}^m_h$, $V=V_h$, $A=A_h^1$. Then $\mathbb{B}_h^m$ is a preconditioner for $\mathbb{A}_h$ satisfying
\begin{equation*}
    \kappa(\mathbb{B}_h^m\mathbb{A}_h)\preccurlyeq1.
\end{equation*}
\end{theorem}


\begin{figure}[tbhp]
\centering
\subfloat[]{\label{torus0}\includegraphics[width=4cm,height=4cm]{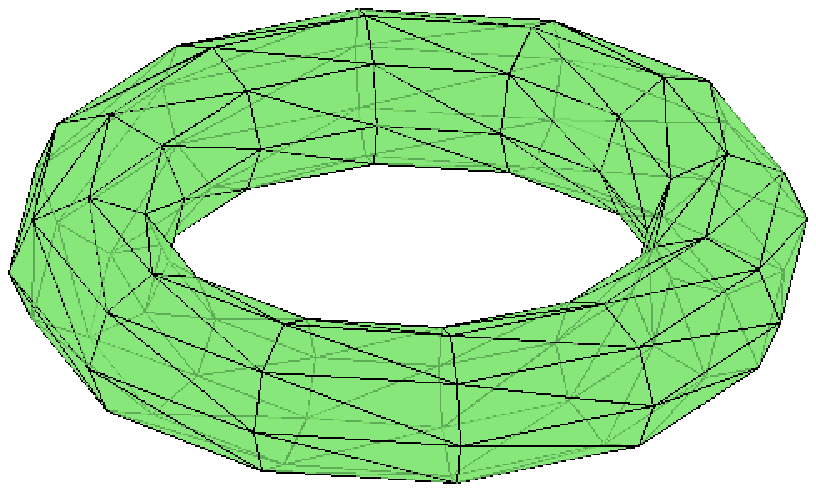}}
\subfloat[]{\label{torus1}\includegraphics[width=4cm,height=4cm]{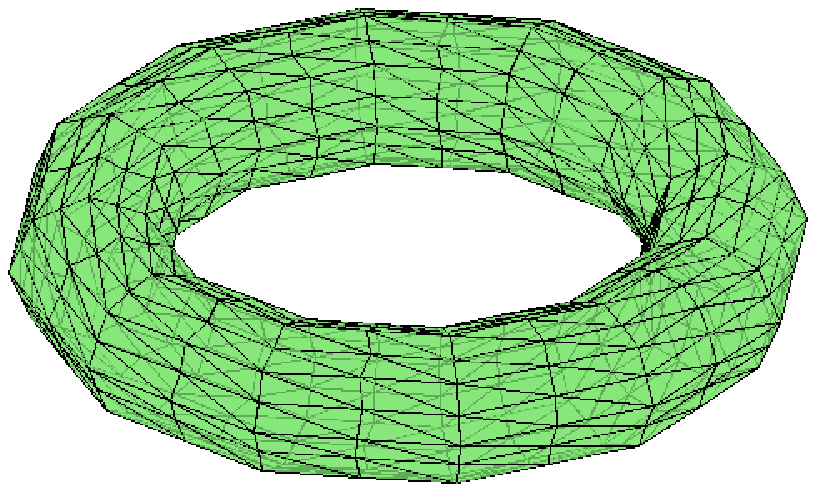}}
\subfloat[]{\label{torus2}\includegraphics[width=4cm,height=4cm]{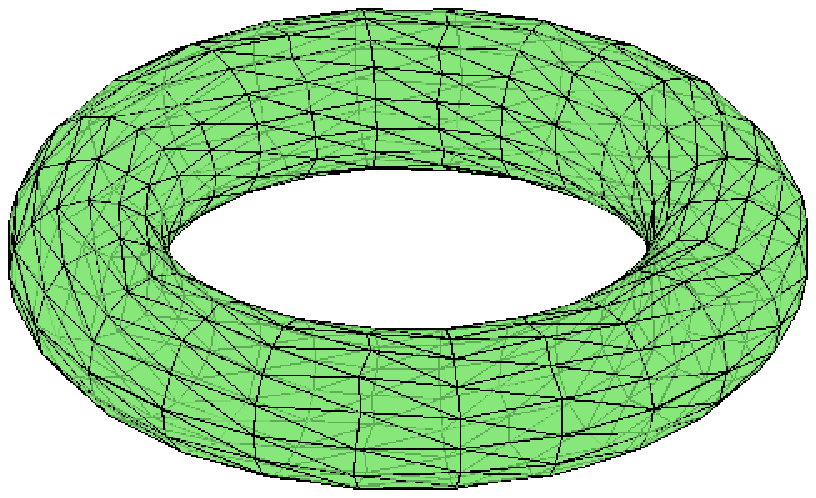}}
\caption{Triangulations of $\mathbb{T}^2$; (a) an initial mesh $\mathbb{T}_0^2$, 192 elements; (b) a mesh on $\mathbb{T}_0^2$, 768 elements; (c) a mesh with vertices on $\mathbb{T}^2$, 768 elements.}
\label{3tori}
\end{figure}

\begin{table}[ht]
\caption{PCG iterations for the linear element on $\mathbb{T}^2$}
\centering
\begin{tabular}{|c|c|c|c|c|c|c|c|c|}
\hline
$N$ & $B_{\text{amg}}^a$&$E_{\text{amg}}^a$ &$B_{\text{amg}}^m$ & $E_{\text{amg}}^m$ 
 &$B_h^a$
& $E_h^a$  & $B_h^m$&  $E_h^m$ \\
\hline
768&29&7.79e-7&9&1.67e-7&12&8.17e-8&9&4.34e-7\\
3072&31&7.05e-7&9&9.67e-7&28&9.92e-7&9&1.82e-7\\
12288&46&7.52e-7&13&5.80e-7&39&7.02e-7&10&2.94e-7\\
49152&65&9.36e-7&17&9.17e-7&45&8.63e-7&11&2.21e-7\\
196608&99&9.48e-7&23&5.62e-7&51&7.19e-7&11&5.69e-7\\
786432&133&9.53e-7&31&6.73e-7&56&8.34e-7&12&2.63e-7\\
\hline
\end{tabular}
\label{T2lineartab}
\end{table}

\begin{table}[ht]
\caption{PCG iterations for the linear element on $\mathbb{S}^3$}
\centering
\begin{tabular}{|c|c|c|c|c|c|c|c|c|}
\hline
$N$ & $B_{\text{amg}}^a$ & $E_{\text{amg}}^a$ &$B_{\text{amg}}^m$ & $E_{\text{amg}}^m$ 
 &$B_h^a$
& $E_h^a$ & $B_h^m$&  $E_h^m$\\
\hline
128&13&6.47e-7&6&2.12e-7&4&8.17e-8&4&6.50e-7\\
1024&18&6.36e-7&9&6.47e-7&12&9.92e-7&7&7.15e-7\\
8192&22&9.87e-7&10&8.22e-7&17&7.02e-7&9&4.13e-7\\
65536&40&9.85e-7&14&5.65e-7&23&8.63e-7&10&8.00e-7\\
524288&73&8.03e-7&20&8.30e-7&27&7.19e-7&12&3.06e-7\\
4194304&104&8.94e-7&26&6.64e-7&31&8.34e-7&13&4.47e-7\\
\hline
\end{tabular}
\label{S3lineartab}
\end{table}

\begin{table}[ht]
\caption{Convergence history of discretization errors on $\mathbb{S}^3$}
\centering
\begin{tabular}{|c|c|c|c|c|c|c|}
\hline
$N$ & $\|u-u_h\|$& order &$\|u-u^{\text{CR}}_h\|$& order & $\|u-u_h^{\text{DG}}\|$&order\\
\hline
128&2.14&&2.09&&2.09&\\
1024&9.37e-1&1.19&9.12e-1 &1.20&9.12e-1&1.20\\
8192&2.82e-1&1.73&2.72e-1&1.75&2.72e-1&1.75\\
65536&7.41e-2&1.93&7.11e-2&1.94&7.11e-2&1.94\\
524288&1.88e-2&1.98&1.80e-2&1.98&1.80e-2&1.98\\
4194304&4.71e-3&2.00&4.51e-3&2.00&4.51e-3&2.00\\
\hline
\end{tabular}
\label{S3convergencetab}
\end{table}

\begin{table}[ht]
\caption{PCG iterations for the CR element on $\mathbb{T}^2$}
\centering
\begin{tabular}{|c|c|c|c|c|c|c|}
\hline
$N$ & $\mathcal{B}_{\text{amg}}^m$ & $\mathcal{E}_{\text{amg}}^m$ 
 &$\mathcal{B}_h^a$
& $\mathcal{E}_h^a$ & $\mathcal{B}_h^m$&  $\mathcal{E}_h^m$\\
\hline
768&10&6.79e-7&37&7.81e-7&10&7.06e-7\\
3072&13&7.24e-7&43&7.13e-7&12&3.73e-7\\
12288&17&4.06e-7&45&9.68e-7&13&5.95e-7\\
49152&22&4.99e-7&49&9.77e-7&18&5.03e-7\\
196608&29&8.21e-7&54&9.62e-7&24&6.13e-7\\
786432&38&9.28e-7&61&8.42e-7&32&5.71e-7\\
\hline
\end{tabular}
\label{T2NCtab}
\end{table}

\begin{table}[tbhp]
\caption{PCG iterations for the DG method on $\mathbb{T}^2$}
\centering
\begin{tabular}{|c|c|c|c|c|c|c|}
\hline
$N$ & $\mathbb{B}_{\text{amg}}^m$ & $\mathbb{E}_{\text{amg}}^m$ 
 &$\mathbb{B}_h^a$
& $\mathbb{E}_h^a$ & $\mathbb{B}_h^m$&  $\mathbb{E}_h^m$ \\
\hline
768&19&5.65e-7&66&9.77e-7&16&7.94e-7\\
3072&25&6.47e-7&69&8.50e-7&18&5.87e-7\\
12288&33&7.32e-7&72&8.57e-7&19&7.41e-7\\
49152&43&9.77e-7&76&8.18e-7&20&8.17e-7\\
196608&59&8.65e-7&80&9.83e-7&24&6.80e-7\\
786432&79&8.73e-7&86&9.83e-7&32&6.75e-7\\
\hline
\end{tabular}
\label{T2DGtab}
\end{table}

\begin{table}[ht]
\caption{PCG iterations for  the CR element on $\mathbb{S}^3$}
\centering
\begin{tabular}{|c|c|c|c|c|c|c|}
\hline
$N$ & $\mathcal{B}_{\text{amg}}^m$ & $\mathcal{E}_{\text{amg}}^m$ 
 &$\mathcal{B}_h^a$
& $\mathcal{E}_h^a$  & $\mathcal{B}_h^m$&  $\mathcal{E}_h^m$ \\
\hline
128&7&5.02e-7&22&8.27e-7&6&6.75e-7\\
1024&9&2.08e-7&27&5.65e-7&7&7.88e-7\\
8192&12&7.53e-7&29&7.59e-7&9&4.19e-7\\
65536&19&5.61e-7&32&8.07e-7&12&5.85e-7\\
524288&24&9.58e-7&36&7.53e-7&17&7.64e-7\\
4194304&34&8.08e-7&44&7.75e-7&23&5.80e-7\\
\hline
\end{tabular}
\label{S3NCtab}
\end{table}

\begin{table}[ht]
\caption{PCG iterations for  the DG method on $\mathbb{S}^3$}
\centering
\begin{tabular}{|c|c|c|c|c|c|c|}
\hline
$N$ & $\mathbb{B}_{\text{amg}}^m$ & $\mathbb{E}_{\text{amg}}^m$ 
 &$\mathbb{B}_h^a$
& $\mathbb{E}_h^a$  & $\mathbb{B}_h^m$&  $\mathbb{E}_h^m$\\
\hline
128&34&6.96e-7&84&8.17e-8&29&9.83e-7\\
1024&38&6.69e-7&104&9.92e-7&32&4.78e-7\\
8192&46&9.24e-7&108&7.02e-7&32&6.12e-7\\
65536&57&8.99e-7&123&8.63e-7&33&8.94e-7\\
524288&75&7.79e-7&134&7.19e-7&36&9.01e-7\\
4194304&105&9.25e-7&147&9.64e-7&39&8.17e-7\\
\hline
\end{tabular}
\label{S3DGtab}
\end{table}
\section{Numerical experiments}\label{secNE}
In this section, we test the performance of several additive and multiplicative preconditioners for the conforming linear, CR and DG discretizations of the problem \eqref{LB} on 2 and 3 dimensional hypersurfaces. In particular, we consider the 2-torus
\begin{equation*}
    \mathbb{T}^2=\left\{x\in\mathbb{R}^3: d_{\mathbb{T}^2}(x)=\sqrt{\big(\sqrt{x_1^2+x_2^2}-R\big)^2+x_3^2}-r=0\right\}
\end{equation*}
with $R=2, r=0.5,$ 
and the unit 3-sphere \begin{equation*}
\mathbb{S}^3=\left\{x\in\mathbb{R}^4: d_{\mathbb{S}^3}(x)=\sqrt{x_1^2+x_2^2+x_3^2+x_4^2}-1=0\right\}.
\end{equation*}
The initial triangulation of $\mathbb{T}^2$ is $\mathbb{T}_0^2$ shown in Fig.~\ref{torus0}. The reference grid sequence on  $\mathbb{T}_0^2$ is constructed by successively quad-refining the initial mesh (see Fig.~\ref{torus1}). Using the signed distance function $d_{\mathbb{T}^2}$, we construct the function $\Phi$ in \eqref{Phi}. Then the true grid hierarchy for $\mathbb{T}^2$ is obtained by mapping reference grid vertices from $\mathbb{T}_0^2$ to $\mathbb{T}^2$ via $\Phi$, see Fig.~\ref{torus2}.

The triangulation of $\mathbb{S}^3$ could not be visualized in $\mathbb{R}^3$. Let $p_1=(1, 0, 0, 0)$, $p_2=(0, 1, 0, 0)$, $p_3=(-1, 0, 0, 0)$, $p_4=(0, -1, 0, 0)$, 
$p_5=(0, 0, 1, 0)$, $p_6=(0, 0, -1, 0)$, $p_7=(0, 0, 0, 1)$, $p_8=(0, 0, 0, -1)$, and $\overline{p_ip_jp_kp_\ell}$ the simplex with vertices $p_i, p_j, p_k, p_\ell$. The initial mesh of $\mathbb{S}^3$ consists of the following simplexes $\overline{p_1p_2p_5p_7}$, $\overline{p_3p_5p_2p_7}$, $\overline{p_3p_4p_5p_7}$, $\overline{p_1p_5p_4p_7}$, $\overline{p_1p_6p_2p_7}$, $\overline{p_3p_2p_6p_7}$, $\overline{p_3p_6p_4p_7}$, $\overline{p_1p_4p_6p_7}$, $\overline{p_8p_1p_2p_5}$, $\overline{p_8p_3p_5p_2}$,  $\overline{p_8p_3p_4p_5}$,\\ $\overline{p_8p_1p_5p_4}$, $\overline{p_8p_1p_6p_2}$, $\overline{p_8p_3p_2p_6}$, $\overline{p_8p_3p_6p_4}$, $\overline{p_8p_1p_4p_6}$. This simplicial mesh is uniformly octa-refined by the algorithm in \cite{Bey2000} to generate a sequence of reference meshes, which are used to construct the true triangluations of $\mathbb{S}^3$ via $\Phi$ based on the signed distance function $d_{\mathbb{S}^3}$. To ensure the correctness of the code in $\mathbb{R}^4$, we compute the discretization errors using the exact solution \begin{equation*}
    u(x)=x_1+2x_2+3x_3+x_4,\quad f=-\Delta_{\mathbb{S}^3}u=3u
\end{equation*}
for \eqref{LB} with $c=0$,
see Table \ref{S3convergencetab} for the convergence history.

In each experiment, the algebraic system of linear equations are solved by the PCG method,  implemented as the function \textsf{pcg}  with error tolerance $10^{-6}$ in MATLAB R2020a. The preconditioner might be the classical AMG algorithm with filtering threshold $\theta=0.025$, two-point interpolation, and exact solve in the coarsest level (see \cite{RugeStuben1987,XuZikatanov2017}) and the corresponding code is available in the iFEM package \cite{iFEM}. By $B_{\text{amg}}^m$, $\mathcal{B}_{\text{amg}}^m$, $\mathbb{B}_{\text{amg}}^m$, we denote the direct AMG V-cycle preconditioners with two pre-smoothing and two post-smoothing steps at each level for the linear nodal element, CR element, and DG stiffness matrices, respectively. Similarly, $B_{\text{amg}}^a$ is the AMG BPX-type multilevel preconditioner for the nodal element stiffness matrix. For the semi-definite problem \eqref{dLB}, \eqref{CRLB}, \eqref{DGLB} with $c=0$, direct AMG preconditioners are constructed based on corresponding positive definite problems with $c=1.$

In each table, the PCG iterative errors based on preconditioners $B_{\text{amg}}^m$, $B_{\text{amg}}^a$, $B_h^m$, $B_h^a$, $\mathcal{B}_{\text{amg}}^m$, $\mathbb{B}_{\text{amg}}^m$ are denoted as  $E_{\text{amg}}^m$, $E_{\text{amg}}^a$, $E_h^m$, $E_h^a$, $\mathcal{E}_{\text{amg}}^m$, $\mathbb{E}_{\text{amg}}^m$, respectively. By $N$ we denote the number of elements in the current mesh.

\subsection{Linear nodal elements} 

For the linear discretization \eqref{dLB} with $c=1$ on $\mathbb{T}^2$ and with $c=0$ on $\mathbb{S}^3$, we implement the multilevel additive preconditioner $B^a_h=\Pi_h\widehat{B}^a_h\Pi_h^\prime$ and  multiplicative preconditioner $B^m_h=\Pi_h\widehat{B}^m_h\Pi_h^\prime$ with $\widehat{B}^a_h$ and $\widehat{B}^m_h$ described in \eqref{Bhahat}, \eqref{mguniform} and Algorithm \ref{wowo}. Here $\widehat{B}^m_h$ and $\widehat{B}^a_h$ utilize  exact solve in the coarsest level. Under the usual nodal basis for $V_h$ and $\widehat{V}_h$, the transfer operators $\Pi_h$, $\Pi_h^\prime$
are simply represented as identity matrices.

It is observed from Tables \ref{T2lineartab} and \ref{S3lineartab} that the direct AMG preconditioners work well for the linear nodal element and the geometric multigrid is more efficient and robust, especially for extremely large systems.

\subsection{CR and DG discretizations}
For the CR \eqref{CRLB} and DG \eqref{DGLB} methods with $c=1$ on $\mathbb{T}^2$ and with $c=0$ on $\mathbb{S}^3$, we test the two-level preconditioners $\mathcal{B}_h^a$, $\mathcal{B}_h^m$ in Theorems \ref{CRpreconditioner} and \ref{CRpreconditioner2} and $\mathbb{B}_h^a$, $\mathbb{B}_h^m$ in Theorems \ref{DGpreconditioner} and \ref{DGpreconditioner2}. In the finest level, the smoother for $\mathcal{B}_h^a$, $\mathbb{B}_h^a$ is the Jacobi relaxation while $\mathcal{B}_h^m$, $\mathbb{B}_h^m$ use two forward and two backward Gauss--Seidel local relaxations. In the coarse level, $\mathcal{B}_h^a$, $\mathbb{B}_h^a$, $\mathcal{B}_h^m$, $\mathbb{B}_h^m$ utilize the AMG V-cycle for the linear nodal element as the approximate solve and  grid hierarchy is not required. 

As shown in Tables \ref{T2NCtab}--\ref{S3DGtab}, all preconditioners work well for CR and DG methods. The direct AMG V-cycle preconditioners $\mathcal{B}_{\text{amg}}^m$, $\mathbb{B}_{\text{amg}}^m$ outperform the corresponding two-level additive $\mathcal{B}_h^a$, $\mathbb{B}_h^a$. However, the semi-analytic two-level multiplicative preconditioners $\mathcal{B}_h^m$, $\mathbb{B}_h^m$ are no worse than direct AMG. At no additional cost, they are obviously more efficient than AMG preconditioners when applied to  the CR method on $\mathbb{S}^3$ and the DG method on $\mathbb{T}^2$ and $\mathbb{S}^3$.

\section{Concluding remarks}\label{secconclusion}
In this paper, we developed optimal FASP preconditioners for the linear nodal element, CR element, and DG method for the second order elliptic equation on hypersurfaces. 
To achieve higher order accuracy on surfaces, it is necessary to use a isoparametric discrete surface $\mathcal{M}_h^p$ under some piecewise polynomial parametrization of degree $p$ with $p\geq2$,   cf.~\cite{Demlow2009}. For example, let $\Phi_h^p$ be the nodal interpolant of $\Phi$ of degree $p$ on $\mathcal{M}_h$. The higher order discrete surface is defined as $\mathcal{M}_h^p=\Phi_h^p(\mathcal{M}_h)$. The isoparametric finite element space of degree $p$ is 
\begin{equation*}
    \mathcal{V}_h^p=\big\{v_h\in H^1(\mathcal{M}^p_h): (v_h\circ\Phi_h^p)|_\tau\in\mathcal{P}_p(\tau)~\forall \tau\in\mathcal{T}_h\big\}.
\end{equation*} 
Then one could numerically solve \eqref{LB} on $\mathcal{M}_h^p$ based on $\mathcal{V}_h^p.$ We note that $\mathcal{V}_h^p$ is naturally connected with the nodal element space $\bar{\mathcal{V}}_h^p$ of degree $p$ on $\mathcal{M}_h$ via the transfer operator $\Pi_h^p: \bar{\mathcal{V}}_h^p\rightarrow\mathcal{V}_h^p$, where $\Pi_h^p(v_h)=v_h\circ(\Phi_h^p)^{-1}$. Let $\bar{A}_h^p$ be the discrete operator of the conforming nodal element of degree $p$ on $\mathcal{M}_h.$  Following the same analysis in Section \ref{seclinear}, $\Pi_h^p\bar{A}_h^p(\Pi_h^p)^\prime$ is a preconditioner for the higher order method on $\mathcal{M}_h^p$. In the next step, one could use the conforming linear element space on $\mathcal{M}_h$ as the auxiliary space to construct a two-level preconditioner for $\bar{A}_h^p$ as in Section \ref{secNCDG}. In the end, the discrete operator $A_h^1$ of linear elements could be further approximated using results in Section \ref{seclinear}.

\bibliographystyle{siamplain}

\end{document}